\documentclass{article}
\usepackage[utf8]{inputenc}
\usepackage{mathtools}
\usepackage{amsfonts,hyperref,url}
\usepackage{amssymb}
\usepackage{amsmath,amsthm}
\usepackage{upgreek }
\usepackage{bbm}
\usepackage{esint}
\usepackage{appendix}

\newcommand{\eq}[1]{\begin{align}#1\end{align}}

\newcommand{\ba}{\begin{eqnarray}}
\newcommand{\ea}{\end{eqnarray}}
\newcommand{\epsi}{\varepsilon}
\def\d{{\rm d}}

\newtheorem{theorem}{Theorem}
\numberwithin{theorem}{section}

\newtheorem{lemma}[theorem]{Lemma}
\newtheorem{corollary}[theorem]{Corollary}

\newtheorem{proposition}[theorem]{Proposition}

\newtheorem*{theorem*}{Theorem}
\newtheorem*{lemma*}{Lemma}
\newtheorem*{corollary*}{Corollary}
\newtheorem*{proposition*}{Proposition}
\newtheorem*{problem*}{Problem}
\newtheorem*{conjecture*}{Conjecture}

\newtheorem{definition}[theorem]{Definition}
\newtheorem{hyp}{Assumption}

\newtheorem{remark}[theorem]{Remark}

\newcommand{\Rr}{{\mathbb{R}}}

\title{A pathwise Ito formula \\for weakly differentiable functions}
\author{Anna ANANOVA and Rama CONT}
\date{2026}

\begin{document}

\maketitle

  \begin{abstract}
We extend F\"ollmer's pathwise It\^o formula to weakly
differentiable functions of continuous paths with finite quadratic
variation along a   sequence of partitions. For each such
path $\omega$, we introduce a path-dependent Sobolev space
$W_{\omega,\pi}^{2}$ defined through smooth approximation of the weak Hessian in a seminorm generated
by discrete weighted occupation measures of the path  $\omega$.
For $F\in W_{\omega,\pi}^{2}$, we construct the pathwise integral
$\int \nabla F(\omega)\,d^{\pi}\omega$ and the covariation
$[\nabla F(\omega),\omega]_{\pi}$, and prove the change-of-variable
formula
$$
F(\omega(t))-F(\omega(0))
=
\int_{0}^{t}\nabla F(\omega(s))\,d^{\pi}\omega(s)
+
\frac12[\nabla F(\omega),\omega]_{\pi}(t).
$$
For Brownian motion, we show that functions in
$W^{2+,p}(\mathbb{R}^{d})\cap W^{2,1}(\mathbb{R}^{d})$ belong almost
surely to the corresponding path-dependent space, outside a polar
exceptional set of starting points. If, in addition,
$F\in W_{\mathrm{loc}}^{1,2}(\mathbb{R}^{d})$, the pathwise integral
agrees with the stochastic It\^o integral, yielding a pathwise version
of the multidimensional F\"ollmer--Protter formula. 
\end{abstract}
\newpage
\tableofcontents
\newpage
\section{Ito formulae with covariation term}

Ito's stochastic change of variable formula is a cornerstone in the development of stochastic analysis  \cite{ikedawatanabe,P05} and has been extended in various directions. For a function $f\in C^2(\mathbb{R}^d)$   of a $\mathbb{R}^d-$valued semimartingale $X$ it takes the standard form
\begin{align}\label{Ito}
f(X_T) = f(X_0) + \int_0^T \nabla f(X_s)dX_s + \frac{1}{2}\int_0^T \langle \nabla^2 f(X_s), d[X]_s\rangle
\end{align}
where equality is almost-sure with respect to the law of   $X$.

An important class of extensions of \eqref{Ito} is to  functions
whose regularity is specified in terms of Sobolev spaces. 
F\"{o}llmer, Protter and Shiryaev \cite{FPS1995}  proved a change of variable formula for functions in $W^{1,2}(\Rr)$ applied to a scalar  Brownian motion $B$:
\begin{align}\label{FPSIto}
f(B_T) = f(B_0) + \int_0^T  f'(B_s)dB_s + \frac{1}{2}[ f'(B_\cdot), B_\cdot](T),
\end{align} where the term involving the second derivative is replaced with the covariation  term $[f'(B), B]$. We refer to \eqref{FPSIto} as an Ito formula with covariation term.
This result was extended by 
F\"{o}llmer and Protter \cite{follmerprotter2000}  to multidimensional Brownian motion, then by Bardina and Jolis \cite{bardina1997} to elliptic diffusions  using  Malliavin calculus  and by
Moret and Nualart \cite{moret2001generalization} to square integrable martingales. 
Eisenbaum \cite{eisenbaum2001ito} gives an alternative formulation of the covariation in terms of semimartingale local time. 
More recently, Ibragimov, Smorodina and Faddeev~\cite{ibragimov2024}
 showed that, in the one-dimensional case,   under $F'\in L^2_{\mathrm{loc}}(\mathbb R)$,
the second-order term 
$$
\int_0^t F''(B_s)\,ds,
$$
where $F''$ is  a distribution, may be defined by regularization as a limit in probability.
The methods used   in these references \cite{bardina1997,FPS1995,follmerprotter2000,ibragimov2024,moret2001generalization} are probabilistic and use properties of Brownian motion and Gaussian estimates in an essential way.

On the other hand, in earlier work \cite{follmer1981}, 
F\"{o}llmer proved a pathwise version of the Ito formula \eqref{Ito} for paths with finite quadratic variation along a sequence of partitions, showing that the Ito formula is essentially an analytical, not probabilistic, change of variable formula for a well-defined class of paths. Extensions of F\"{o}llmer's pathwise approach 
for functions with Sobolev regularity have been studied  by Wuermli \cite{wuermli1980},
Bertoin \cite{bertoin1987}  Davis, Obloj and Siorpaes \cite{davis2014}, Perkowski and Promel \cite{perkowski2015}.

\paragraph{Contribution}

We use F\"ollmer's pathwise approach to derive change of variable formulas for weakly differentiable functions of paths with finite quadratic variation, obtaining pathwise analogues of the results of
F\"{o}llmer, Protter and Shiryaev (1995) \cite{FPS1995} and F\"{o}llmer \& Protter (2000) \cite{follmerprotter2000}.
Our construction separates the 
analytic mechanism behind the change-of-variable formula from the
probabilistic estimates needed to verify the required Sobolev trace
conditions for Brownian paths. These are objects of two separate theorems, which constitute our main results.

Given a path $\omega\in C([0,T], \Rr^d)$ with finite quadratic variation along a sequence of partitions $\pi=\{\pi_n\}$, we obtain an Ito-type change of variable formula for a subspace of the Sobolev space $W^{2, 1}(\Rr^d)$ associated with $\omega$. To define this space, we introduce a sequence of  measures $\mathcal{L}^{\pi_n}_{\omega}$ of $\omega$ which  are discrete approximations to the weighted occupation time measure $\mu_{\omega}$:
\[
\mu_{\omega}(A)={\rm tr}\left( \int_0^T 1_A(\omega(t)) d[\omega](t)\ \right).
\]
and define a space $\mathcal{W}^{2}_{\omega, \pi}\subset W^{2, 1}(\Rr^d)$ of functions with an integrability property with respect to the sequence of measures $\mathcal{L}^{\pi_n}_{\omega}$ (Definition \ref{def:W-omega-pi}).

Our first main result  (Theorem \ref{theorem:weakchangeofvar}) is a     change of variable formula  for functions in $\mathcal{W}^{2}_{\omega, \pi}$ computed along paths of finite quadratic variation.
Our second main result is to show that, when $\omega$ is a typical Brownian path we have
$$
W^{2+,p}(\mathbb{R}^d)\cap W^{2,1}(\mathbb{R}^d)
\subset W_{\omega,\pi}^2
$$
almost surely outside a polar set of starting points (Theorem \ref{lemma:Sobinc}).
Combined in Proposition \ref{prop.FP}, these results yield a pathwise derivation of the probabilistic covariation formulas of \cite{follmerprotter2000} and \cite{FPS1995}, but also disentangle the probabilistic ingredients in these results,  which rely on Brownian motion and its Gaussian properties, from the analytical ingredients which are valid more generally for all paths with finite quadratic variation.

\subsection{Related work}

Davis, Ob{\l}\'oj and Siorpaes \cite{davis2014} developed a
one-dimensional pathwise stochastic calculus based on discrete local
times. They characterize the weak convergence of these discrete local
times in $L^r$ and obtain pathwise It\^o--Tanaka formulae for
functions in the dual Sobolev class, as well as for differences of
convex functions under additional regularity assumptions on the local
time. Cont and Perkowski \cite{cp2018} and Cont and Jin \cite{contjin2024} extended this approach to the functional setting and to paths with finite p-th variation with arbitrary $p>1$, defining local time of order $p$ as a limit of discrete weighted local times.

Our approach is related but different. In dimension one, the
functions $L_{\omega}^{\pi_n}$ introduced in Section~4 are, up to a
factor of one half, the sum of the discrete local times of
Davis--Ob{\l}\'oj--Siorpaes for the path and its time reversal.
However, we do not assume that these functions converge to a pathwise
local time. Instead, they serve as majorizing densities used to define
the path-dependent seminorm $\|\cdot\|_{W^0_{\omega,\pi}}$. This
allows us to formulate a multidimensional change-of-variable formula
whose second-order term is the covariation
$$
[\nabla F(\omega),\omega]_{\pi}.
$$
Whenever both approaches apply in dimension one, this covariation
coincides with the corresponding local-time pairing:
$$
[f'(\omega),\omega]_{\pi}(t)
=
\int_{\mathbb{R}}L_t(u)f''(u)\,du.
$$  
Here we work directly with the sequence of discrete weighted occupation measures and do not require the discrete local times to converge.
\subsection{Outline}
We recall in Section \ref{sec:Sobspace}  some properties of Sobolev functions, Bessel potentials and corresponding polar sets. In Section \ref{sec:forC1functions}, we derive the change of variable formula with covariation term for $f\in C^2_b$. In Section \ref{sec:loctime},   we introduce the majorizing measures $\mathcal{L}^{\pi_n}_{\omega}$ and discusses some of their properties. 
In Section \ref{sec:forSoblv}, we define the spaces $\mathcal{W}^{2}_{\omega, \pi}$ which provide the setting for our main result.
Section \ref{sec:main} contains our main results: Theorem  \ref{theorem:weakchangeofvar} is the change of variable formula,
Theorem \ref{lemma:Sobinc} justifies the definition of the space $W_{\omega,\pi}^2$ by
showing that, when $\omega$ is a typical Brownian path, it contains every
$W^{2+s,p}(\mathbb{R}^d)\cap W^{2,1}(\mathbb{R}^d)$
 outside a polar set of starting points; Proposition \ref{prop.FP} recovers a version of the F\"ollmer-Protter formula \cite{follmerprotter2000}. 
 
 Appendix~\ref{sec:appA}
establishes a uniform $L^q$ bound for discrete
quadratic-variation densities of Brownian paths, used in the proof of Theorem~\ref{lemma:Sobinc}. Appendix \ref{sec:appB} proves some technical lemmas used in the proofs.
Appendix~\ref{sec.Proof54} gives the proof of Theorem~\ref{lemma:Sobinc}.

\section{Some properties of Sobolev spaces}\label{sec:Sobspace}

\subsection{Pointwise Inequalities for Sobolev Functions}

For $ k\in \mathbb {N} ,p\in [1,\infty ],$ denote by $W^{k,p}(\Rr^d)$  the Sobolev space of functions $f\in L^p(\Rr^d)$ such that for every multi-index $\alpha=(\alpha_1, \ldots, \alpha_d)\in \mathbf{N}^d$ with $|\alpha|:=\alpha_1\,+ \ldots\,+\, \alpha_d\leq k$, 
\[
D^{\alpha }f=\frac {\partial ^{|\alpha |}f} { \partial x_{1}^{\alpha _{1}} \dots \partial x_{d}^{\alpha _{d}} }\in L^p(\Rr^d).
\]
In what follows we denote $\nabla^m f$ the vector with components $D^{\alpha }f,\, |\alpha|= m$, and we identify $\nabla^2 f$  with the matrix
 $[\frac {\partial ^{2}f} { \partial x_{i} \partial x_{j}}]_{i, j=1}^d.$

We write
\[
W^{2+,p}(\mathbb{R}^d)
:=
\bigcup_{s\in(0,1)}W^{2+s,p}(\mathbb{R}^d).
\]
Equivalently, $f\in W^{2+,p}(\mathbb{R}^d)$ if and only if
$f\in W^{2+s,p}(\mathbb{R}^d)$ for some $s\in(0,1)$.

%
%
%

We recall some results on the choice of  natural representatives of a function for $f\in W^{k, p}$ and its derivatives, and the properties of such natural representatives. As  in \cite{BH},  we define these natural representatives as
 \eq{\label{2eq:values}
 D^{\alpha}f(x)=\lim\sup_{r \to 0}\fint_{B(x, r)}  D^{\alpha}f(y)\d y,\quad |\alpha|\leq k.
 }

Let $Q\subset \Rr^d$ be either a cube or a ball. For $f\in L^1_{loc}(Q)$ we denote
\[
f_{Q}:=\frac 1 {|Q|}\int_Q f dx.
\]
The following special case of Theorem 1 in \cite{BH} is useful:
\begin{theorem}\label{theorem:Sobinq}
There exists a constant $C_d$ such that if $f\in W^{1, 1}(\Rr^d)$ and its derivatives are defined at every point by \ref{2eq:values}, and $x, y\in\Rr^d$ be such that $|f(y)|, |f(x)| <+\infty.$ If $x, y\in Q$ then
\[
|f(y)- f(x) |\leq C_d\left(\int_Q \frac {|\nabla f(z)|} {|x- z|^{d-1}} dz\,+\, \int_Q \frac {|\nabla f(z)|} {|y- z|^{d-1}} dz\right).
\]
Similarly for $f\in W^{2, 1}(\Rr^d)$, and $x, y\in Q$ with $|f(y)|, |f(x)|, |\nabla f(x)| <+\infty$ 
\[
|f(y)- f(x) - \nabla f(x) (y-x)|\leq C_d\left(\int_Q \frac {|\nabla^2 f(z)|} {|x- z|^{d-2}} dz\,+\, \int_Q \frac {|\nabla^2 f(z)|} {|y- z|^{d-2}} dz\right).
\]
\end{theorem}

Recall the definition of Hardy-Littlewood maximal functions
\[
M_Rf (x):=\sup _{r<R} \fint_{B(x, r)} |f(y)| \d y,\quad Mf(x):= M_{\infty}f(x)=\sup _{r>0} \fint_{B(x, r)} |f(y)| \d y.
\]

\begin{lemma}\label{lemma:hedberg}
For $\alpha>0$ there exists a constant $C_{\alpha, d}>0$ such that for every $f\in L^1(Q)$ and all $x\in Q$
\[
\int_Q \frac {|f(y)|} {|x- y|^{d-\alpha}} \d y \leq C_{\alpha, d} ({\rm diam}\, Q)^{\alpha}M_{{\rm diam}\, Q}f(x).
\]
\end{lemma}

As a consequence of Theorem \ref{theorem:Sobinq} and  Lemma \ref{lemma:hedberg}, we have the following:
\begin{theorem}\label{theorem:Sobineq}
There exists a constant $C_d$ such that if $f\in W^{1, 1}(\Rr^d)$ is defined at every point by \ref{2eq:values}, and $x, y\in\Rr^d$ be such that $|f(x)|, |f(y)|<+\infty.$ Then
\[
|f(y)-f(x)|
\leq
C_d\left(
M_{|x-y|}(|\nabla f|)(x)
+
M_{|x-y|}(|\nabla f|)(y)
\right)|x-y|,
\]
There exists a constant $C_d$ such that if $f\in W^{2, 1}(\Rr^d)$ and its derivatives are defined at every point by \ref{2eq:values}, and $x, y\in\Rr^d$  such that $|f(y)|, |f(x)|, |\nabla f(x)| <+\infty,$ 
\[
|f(y)-f(x)-\nabla f(x)(y-x)|
\leq
C_d\left(
M_{|x-y|}(|\nabla^2 f|)(x)
+
M_{|x-y|}(|\nabla^2 f|)(y)
\right)|x-y|^2.
\]
\end{theorem}

The proof of the estimates in Theorem \ref{theorem:Sobineq} and their higher dimensional analogs can also be found in \cite{Boj06}, where the author proves in addition that these type of inequalities completely characterize Sobolev spaces.

\subsection{Bessel Characterization of Sobolev Spaces}
\label{sec:Bessel}

We present here briefly an interesting characterization of Sobolev spaces through `Bessel potentials'.
For detailed discussions of Bessel potentials and their properties we refer to \cite{stein1970}, \cite{fukushima1993}.

\begin{definition}[Bessel Kernel]
The Bessel kernel $g_{\alpha}\colon \Rr^d\to \Rr,\, \alpha>0$ is defined by its Fourier transform
\[
\hat g_{\alpha}(\xi):= \frac 1 { (1+4\pi ^{2}\vert \xi \vert ^{2})^{\alpha/2}}.
\]
where the Fourier transform is defined as
\[
\hat {f}(\xi )={\mathcal {F}}(f)(\xi)=\int _{\mathbb {R} ^{d}}f(x )e^{-2\pi i\, x\cdot \xi }\,\d x.
\]
Equivalently, one has
\[
 (I-\Delta )^{-\alpha/2}f=g_{\alpha}\ast f,
\]
where $\Delta:= \sum_{i=1}^{n}{\frac {\partial ^{2}}{\partial x_{i}^{2}}}$ is the Laplace operator. 

\end{definition}

For $0<\alpha \leq d$ Bessel kernel $g_{\alpha}$ has the following behavior at the origin:
\eq{\label{2eq:Bessassymp}
g_{\alpha}(x) =\frac 1 {\gamma(\alpha) |x|^{d-\alpha}} + o\left(\frac 1 { |x|^{d-\alpha}} \right), \quad 0<\alpha <d \nonumber\\
\\
g_{d}(x)={\frac {1}{2^{d-1}\pi ^{d/2}}}\ln {\frac {1}{\vert x\vert }}(1+o(1)),\,  \text{ as } |x|\to 0.\nonumber
}

%
%
%
%
%

To define the Bessel potential spaces, note that for any function $f\in L^p(\Rr^d),$ where $1\leq p \leq \infty$ the convolution $g_{\alpha}\star f$ is well defined as a function on $L^p(\Rr^d)$, since
\[
\|g_{\alpha}\star f\|_{L^p(\Rr^d)} \leq \|g_{\alpha}\|_{L^1(\Rr^d)} \|f\|_{L^p(\Rr^d)} .
\]

\begin{definition}[Bessel Space]
For $p\in [1, \infty]$ and $\alpha> 0$ the Bessel space $\mathcal{L}^p_{\alpha}(\Rr^d)$ is defined by
\[
\mathcal{L}^p_{\alpha}(\Rr^d):=\{\, g_{\alpha}\star f\colon\, f\in L^p(\Rr^d) \,\},
\]
where $g_0:= \delta_0$, so that $g_{0}\star f =f$.  The space $\mathcal{L}^p_{\alpha}(\Rr^d)$ is endowed with the norm
\[
\|F\|_{\mathcal{L}^p_{\alpha}(\Rr^d)} :=\|f\|_{L^p(\Rr^d)}
\]
where $f\in L^p(\Rr^d)$ such that $F=g_{\alpha}\star f$.
\end{definition}

\begin{theorem}[\cite{stein1970}]\label{theorem:Bessel}
Let $k$ be a positive integer and $1<p<+\infty$, then 
\[
\mathcal{L}^p_{k}(\Rr^d) = W^{k, p}(\Rr^d).
\]
In the sense that $F\in \mathcal{L}^p_{k}(\Rr^d)$ if and only if $F\in W^{k, p}(\Rr^d)$ and the norms $\|\cdot \|_{\mathcal{L}^p_{k}(\Rr^d)}$ and $\|\cdot\|_{W^{k, p}(\Rr^d)}$ are equivalent.
\end{theorem}

\begin{definition}[Bessel capacity]\label{def:besscap}
For $p\in [1, \infty]$ and $\alpha>0$  the Bessel $(\alpha, p)$-capacity of a set $E\subset \Rr^d$ is defined as
\[
B_{\alpha, p}(E):=\inf\{\|h\|_{L^p}^p\colon\, g_{\alpha}\star h\geq  1_E,\, h\geq 0 \}.
\]
We call a set $E$ $(\alpha, p)$-polar if $B_{\alpha, p}(E)=0.$
\end{definition}
It is known that the Bessel potential is countably subadditive (see \cite{fukushima1993}), thus any union of countable number of  $(\alpha, p)$-polar sets is also $(\alpha, p)$-polar.

Hereafter, we denote $\Omega:=C([0, T], \Rr^d)$, let $B$ be the canonical coordinate process on $\Omega$ and $\mathbb{P}_x$ be the law of Brownian motion starting at $x$. A probabilistic definition of a polar set, associated with Brownian motion, is the following \cite{fukushima1993}:  
\begin{definition}[Polar set]\label{chp2def:polar}
A set $E\subset \Rr^d$ is called polar if
\[
\mathbb{P}_x\left( B(t)\in E \text{ for some } t>0\right)=0,\quad \forall x\in \Rr^d.
\]
\end{definition}

It is known that $E$ is polar if and only if $E$ is $(1, 2)$-polar \cite[p. 25]{fukushima1993}.

Finally we recall the following important theorem \cite[Corollary 4]{BH}:
\begin{theorem}[\cite{BH}]\label{theorem:Lebesgue}
Let $k\in\mathbb N$ and $1<p<\infty$. The exceptional set of Lebesgue points of any $F\in W^{k,p}(\Rr^d)$ is $(k, p)$-polar.
\end{theorem}

\section{Change of variable formula with covariation term for smooth functions} \label{sec:forC1functions}

We first prove a pathwise change of variable formula with covariation term for twice differentiable functions. 
The setting is the same as in F\"ollmer (1981) \cite{follmer1981}.

Let us recall the definition of pathwise quadratic variation along a sequence of time partitions \cite{follmer1981,ananova2017} in the continuous case. Consider a sequence $\pi_n = \{0 = t_0^n < t_1^n < ... < t^n_{m(n)} = T\}$ of partitions of $[0, T]$. A continuous path $x \in C^0([0, T], \mathbb{R})$ is said to have finite \textbf{quadratic variation} along the sequence of partition $\pi = (\pi_n)_{n \geq 1}$ if for all $ t \in [0, T]$ the limit
\begin{align}
[x]_\pi(t) := \lim_{n \rightarrow \infty} \sum_{t^n_{i + 1} \leq t}\left(x(t_{i+1}^n) - x(t_i^n)\right)^2 < \infty
\end{align}
exists and the    function $t\mapsto [x]_\pi(t)$  is a continuous  increasing function. \\
We denote the set of paths with finite quadratic variation along $\pi$ by $Q^\pi([0, T], \mathbb{R})$. \\
The definition in the case where the path $x$ is $d$-dimensional is slightly different. A path $x = (x^1, ..., x^d) \in C^0([0, T], \mathbb{R}^d)$ is said to have finite quadratic variation along $\pi$ if $x^i \in  Q^\pi([0, T], \mathbb{R})$ for all $i = 1, \ldots, d$ and if $x^i + x^j \in Q^\pi([0, T], \mathbb{R})$ for all $i, j = 1, \ldots, d$. If this is the case, then we have:
\begin{align*}
\sum_{t_k^n \in \pi_n, t_{k+1}^n \leq  t} \left(x^i(t_{k+1}^n) - x^i(t_k^n)\right) \cdot \left(x^j(t_{k+1}^n) - x^j(t_k^n)\right) 
\\
\xrightarrow{n \to\infty}[x]_{ij}(t) = \frac{[x^i+x^j](t) - [x^i](t) - [x^j](t)}{2}.
\end{align*}
The matrix function $[x]_\pi : [0, T] \rightarrow S_d^+$, whose elements $[x]_{ij}$ are defined as above, is called the quadratic covariation of the path $x$.

In the sequel we will fix a sequence of partitions $\pi$ and assume that 
\eq{\label{2eq:oscassump}
osc(\omega, \pi_n) :=\max_{i=\overline{0,\, m(n)-1}}\sup_{t\in [t^n_i, t^n_{i+1}]}|\omega(t) - \omega(t^n_i)| \rightarrow 0.
}

Alternatively we can assume $|\pi_n| \rightarrow 0$, where $|\pi_n| := \sup\{|t_{i + 1}^n - t_i^n|, i = 1, ..., m(n)-1 \} $ denotes the mesh size of the partition $\pi_n$. This will imply \eqref{2eq:oscassump} due to the continuity of $\omega.$

\begin{proposition} \label{proposition:C1covaritatin}Let $\omega \in C^0([0,T], \Rr^d)\cap Q^\pi([0, T], \mathbb{R}^d)$ satisfying \eqref{2eq:oscassump}. For any $f\in C^1_b(\mathbb{R}^d, \Rr^d)$ the limit
\begin{align*}
 [ f(\omega), \omega]_{\pi}(t):=\lim_{n\to\infty}\sum_{k=0}^{m(n)-1} \left( f(\omega(t^n_{k+1}\wedge t)) - f(\omega(t^n_k\wedge t))\right)\cdot
  \left(\omega(t^n_{k+1}\wedge t)  - \omega(t^n_k\wedge t)\right)
\end{align*}
exists, and is equal to the Riemann–Stieltjes integral
\[
\int_0^t \langle \nabla f(\omega(u)),\, d[\omega]_{\pi}(u)\rangle.
\]
\end{proposition}

\begin{proof}
Let $B$ be a closed ball containing $\omega([0, T])$, since $\nabla f$ is continuous on $B$, it is also uniformly continuous there, this implies that
\[
\rho_{\nabla f}(\delta):= \sup_{|x-y|\leq\delta,\, x, y\in B} |\nabla f(x) -\nabla f(y)| \xrightarrow{\delta\to 0+} 0.
\]
The Newton-Leibniz change of variable formula yields
\begin{align*}
 f(\omega(t^n_{k+1})) - f(\omega(t^n_k))
 &= \nabla f(\omega(t^n_k))\cdot \delta^n_k\omega \, 
\\
+&\,
\underbrace{ \int_{0}^{1} \left(\nabla f(\omega(t^n_k)+u \delta^n_k\omega ) - \nabla f(\omega(t^n_k))\right)\cdot \delta^n_k\omega \ \d s}_{r^n_k}
\end{align*}
where $\delta_k^n\omega = \omega(t^n_{k+1}) - \omega(t^n_k)$. 
Hence, we have
\begin{align*}
\sum_{t^n_k\leq t} \left( f(\omega(t^n_{k+1})) - f(\omega(t^n_k))\right)\cdot \left(\omega(t^n_{k+1})  - \omega(t^n_k)\right)\\
 = \underbrace{\sum_{t^n_k\leq t} \langle\nabla f(\omega(t^n_k)),\, \delta^n_k\omega\, ^t\delta^n_k\omega\rangle}_{I_1^n(t)}\, +\, \underbrace{\sum_{t^n_k\leq t} \delta^n_k\omega\cdot r^n_k}_{I^n_2(t)}.
\end{align*}
We have $I^n_1(t)\xrightarrow{n \to \infty} \int_0^t \langle \nabla f(\omega(u)),\, d[\omega]_{\pi}(u)\rangle,$ since $t\mapsto \nabla f(\omega(t))$ is continuous and the measures $\mu_{_{n, \omega}}:=\sum_{t^n_k\leq t}\delta^n_k\omega\, ^t\delta^n_k\omega\, \delta_{t^n_k}$ converge weakly to $[\omega]_{\pi}(t),$ for details see e.g. \cite{CF10B}.

Thus it remains to prove $I^n_2(t)\xrightarrow{n \to \infty}0$, to do that we note that by convexity $\omega(t^n_k)+u \delta^n_k\omega\in B$, so
\[
\left|\nabla f(\omega(t^n_k)+u \delta^n_k\omega ) - \nabla f(\omega(t^n_k)\right|\leq \rho_f(|\delta^n_k\omega |),
\]
so
\[
|\delta^n_k\omega\cdot r^n_k|\leq \rho_f(|\delta^n_k\omega |)|\delta^n_k\omega|^2. 
\]
Consequently
\[
|I^n_2(t)|\leq \sum_{\pi_n}|\delta^n_k\omega\cdot r^n_k|\leq \rho_f(|osc(\omega, \pi_n)|)\sum_{\pi_n}|\delta^n_k\omega|^2 \to 0.
\]
\end{proof}

\begin{corollary}\label{cor:Itocovformula}
 Let $\omega \in C([0,T], \Rr^d)\cap Q^\pi([0, T], \mathbb{R}^d)$ and $f\in C^2_b(\mathbb{R}^d, \Rr)$. Then the following limit change of variable formula holds
$$
f(\omega(T)) = f(\omega(0)) + \int_0^T \nabla f(\omega)d^\pi\omega + \frac{1}{2}[\nabla f(\omega), \omega]_\pi(T).
$$
where the integral $ \int_0^T \nabla f(\omega)d^\pi\omega$ is defined as the limit of left Riemann sums along the sequence of partitions $(\pi_n)_{n\geq 1}$:
\[
 \int_0^T \nabla f(\omega)d^\pi\omega:=\lim_{n\to\infty}\sum_{\pi_n} \nabla f(\omega(t^n_k))\cdot \left(\omega(t^n_{k+1}) - \omega(t^n_k)\right).
\]
\end{corollary}

\begin{proof}
By F\"{o}llmer-It\^{o} change of variable formula (\cite{follmer1981})
\[
f(\omega(T)) = f(\omega(0)) + \int_0^T \nabla f(\omega)d^\pi\omega + \frac{1}{2}\int_0^T \langle\nabla^2 f(\omega),\, d[\omega]_\pi(t)\rangle.
\]
It remains to note that by Proposition \ref{proposition:C1covaritatin}
\[
[\nabla f(\omega), \omega]_\pi(T)=\int_0^T \langle\nabla^2 f(\omega(t)),\, d[\omega]_\pi(t)\rangle.
\]
\end{proof}

\section{Majorizing measures and function spaces}
In this section we associate with a path with finite quadratic variation $\omega \in C^0([0,T], \Rr^d)\cap Q^\pi([0, T], \mathbb{R}^d)$  a sequence of {\it weighted occupation measures} and use them to define a function space which provides the setting for our results in the next section.
For a continuous function $\phi\in C_c(\Rr^d)$, we denote $${\rm osc}(\phi, \delta):=\sup\{\|\phi(x)-\phi(y)\| \colon |x-y|\leq \delta\}.$$ Then $osc(\phi, \delta)\to 0$ as ${\delta \to 0}$. 
\subsection{Discrete approximations of weighted occupation measures}\label{sec:loctime}
Let $\pi_n := \{0 = t_0^n < t_1^n < ... < t^n_{m(n)} = T\}$ be a sequence of partitions of $[0, T]$, and  $\omega\in C[0, T]$ be path with finite quadratic variation $[\omega]_{\pi}$ along the partitions $\pi_n$ with \eqref{2eq:oscassump}.

Motivated by the estimates of Theorem \ref{theorem:Sobinq},  we introduce the following sequence of functions in $L^1(\Rr^d)$, which resemble the approximations of semimartingale local time in dimension $d=1$ defined in \cite{bertoin1987,perkowski2015,contdas2023,cp2018,contjin2024,davis2014}:
\begin{equation}
    \mathcal{L}^{\pi_n}_{\omega}(x):= \sum_{i=0}^{m(n)-1} \underbrace{c_{d} \left(\frac 1 {|x - \omega(t^n_i)|^{d-2}}\, +\, \frac 1 {|x - \omega(t^n_{i+1})|^{d-2}}\right) 1_{Q^n_i}(x)}_{\mathcal{L}^{n, i}_\omega(x)}
\end{equation}
If $\omega(t_{i+1}^n)=\omega(t_i^n)$, we set
$L_{\omega}^{n,i}\equiv 0$.

Here, $Q^n_i:=\left\{x\colon \left|x- \frac{\omega(t^n_i)\,+\,\omega(t^n_{i+1})}{2}\right|\leq \frac{|\omega(t^n_{i+1})\,-\,\omega(t^n_i)|}2\right\}$, and the constant $c_d$ is chosen such that 
$$
c_d
:=
\left[
\int_{B(0,1/2)}
\left(
\frac{1}{\left|x-\frac12 e_1\right|^{d-2}}
+
\frac{1}{\left|x+\frac12 e_1\right|^{d-2}}
\right)\,dx
\right]^{-1}
=
\frac{2}{|B(0,1)|}.
$$
so that by homogeneity $\int_{\Rr^d} \mathcal{L}^{n, i}_{\omega}(x) dx=|\omega(t^n_{i+1})\,-\,\omega(t^n_i)|^2$ and $$\int_{\Rr^d} \mathcal{L}^{\pi_n}_{\omega}(x) dx= \sum_{i=0}^{m(n)-1} |\omega(t^n_{i+1})\,-\,\omega(t^n_i)|^2.$$
\begin{remark}
In the one dimensional case $d=1$ the balls $Q^n_i$ reduce to intervals $[\omega(t^n_i), \omega(t^n_{i+1})]$, and the function $\mathcal{L}^{\pi_n}_{\omega}$ coincides (up to a multiplicative constant) with the sum of the discrete local times of the path $\omega$  along the partition $\pi_n$ and reversed path $\stackrel{\leftarrow}{\omega}= \omega(T- \cdot)$ along the reversed partition $\stackrel{\leftarrow}{\pi}_n:=\{ T- t^n_i\colon t^n_i\in \pi_n \}$ (see \cite{davis2014}).
\end{remark}

The following lemma confirms the relation  with the quadratic variation-weighted occupation measure:
\begin{lemma}\label{lemma:localtime}
For any path $\omega\in Q_{\pi}([0, T], \Rr^d)$ the sequence $\mathcal{L}^{\pi_n}_{\omega}$ converges weakly to the weighted occupation  measure $\mu^T_{\omega}$ defined by 
\[
\forall A\in {\cal B}(\mathbb{R}^d),\qquad \mu^T_{\omega}(A):= {\rm tr}\int_0^T 1_A(\omega(t)) \d [\omega]_{\pi} (t).
\]
\begin{align*}\forall \phi\in C_c(\Rr^d),\qquad 
\int_{\Rr^d} \phi(x) \mathcal{L}^{\pi_n}_{\omega}(x) \d x  \xrightarrow{n \to \infty} \int_{\Rr^d} \phi(x) d\mu^T_{\omega}(x)
= {\rm tr}\int_0^T (\phi\circ\omega\,) \d[\omega]_{\pi}.
\end{align*}
\end{lemma}

\begin{proof}
First note that on the ball $Q^n_i$ function $\phi$ differs from its value at $\omega(t^n_i)$ at most by 
$$osc(\phi, |\omega(t^n_{i+1})\,-\,\omega(t^n_i)|)\leq osc(\phi, \epsi_n),\text{ where }\epsi_n:=osc(\omega, \pi_n).$$ Thus we have
\eq{
\left|\int_{\Rr^d} \phi(x) \mathcal{L}^{\pi_n}_{\omega}(x\,) \d x - \sum_{i=0}^{m(n)-1}\int_{\Rr^d}  \mathcal{L}^{n, i}_{\omega}(x) \phi(\omega(t^n_i))\,  \d x \right|\nonumber\\ \leq
\int_{\Rr^d} \sum_{i=0}^{m(n)-1} \left|\phi(x)  - \phi(\omega(t^n_i)) \right|  \mathcal{L}^{n, i}_{\omega}(x)\, \d x\nonumber\\
\leq osc(\phi, \epsi_n) \int_{\Rr^d} \sum_{i=0}^{m(n)-1}\mathcal{L}^{n, i}_{\omega}(x)\, \d x=osc(\phi, \epsi_n) \int_{\Rr^d} \mathcal{L}^{\pi_n}_{\omega}(x) \d x\nonumber\\
=osc(\phi, \epsi_n) \sum_{i=0}^{m(n)-1} |\omega(t^n_{i+1})\,-\,\omega(t^n_i)|^2 \xrightarrow{n\to \infty}  0.\nonumber
}
It remains to note
\begin{align*}
\sum_{i=0}^{m(n)-1}\int_{\Rr^d}  \mathcal{L}^{n, i}_{\omega}(x) \phi(\omega(t^n_i))  \d x =
\sum_{i=0}^{m(n)-1}\phi(\omega(t^n_i)) |\omega(t^n_{i+1})\,-\,\omega(t^n_i)|^2\\
 \xrightarrow{n\to \infty}  \int_0^T \phi(\omega(t)) \d tr [\omega]_{\pi} (t).
\end{align*}
\end{proof}

\subsection{Sobolev trace along quadratic-variation paths}\label{sec:forSoblv}
We introduce a family of seminorms using the discrete weighted occupation measures $\mathcal{L}^{\pi_n}_{\omega}$:
\[
\|f\|_{_{\mathcal{W}^{0}_{\omega, \pi}}}:= \limsup_{n \to \infty}\int_{\Rr^d} |f(x)| \mathcal{L}^{\pi_n}_{\omega}(x) \d x,
\]
\begin{definition}[ Pathwise Sobolev space $\mathcal{W}^{2, \pi}_{\omega}$]\label{def:W-omega-pi}
Define   $\mathcal{W}^{2, \pi}_{\omega}$ to be the space of functions $F\in W^{2,1}(\Rr^d)$ for which there exists a sequence $(F_m)_{m\geq 1}$ such that
\eq{\label{2eq:W2pomega}
  F_m\in C^2_b(\Rr^d)\cap W^{2, 1}(\Rr^d)\qquad \qquad  \nonumber\\
F_m(\omega(t)) \to F(\omega(t)),\, \forall t\in [0, T],\quad
\|\nabla^2 F_m -\nabla^2 F\|_{\mathcal{W}^{0}_{\omega, \pi}}\xrightarrow{m\to \infty} 0.
}
\eq{\label{def:W2pdef}
\mathcal{W}^{2}_{\omega, \pi}:=\Big\{F\in W^{2,1}(\Rr^d)\colon\, \exists (F_m)_{m\geq 1}\  \text{\rm satisfying } \eqref{2eq:W2pomega}\Big\}.
}
\end{definition}
$W_{\omega,\pi}^2$ is a Sobolev trace domain along $(\omega,\pi)$.   It is not the range of a classical
trace operator; rather, it is the domain of Sobolev functions whose
values along $\omega$ and whose second-order contribution along the
quadratic-variation occupation measures can be defined consistently
by smooth approximation.

It is immediate that
$$
C_b^2(\mathbb{R}^d)\cap W^{2,1}(\mathbb{R}^d)
\subset W_\omega^{2,\pi}.
$$
The following remark further elucidates the structure of this space in the case where the
weighted occupation measure admits an $L^r$ density.
\begin{remark}[Case of paths with weighted occupation density]
{\em 
Let
$$
\omega\in C([0,T];\mathbb{R}^{d})
\cap Q_{\pi}([0,T];\mathbb{R}^{d}),
$$
and suppose that there exists a
nonnegative function
$
\ell^{\omega,\pi}_{T}\in L^{r}(\mathbb{R}^{d})
$
such that
$$
L_{\omega}^{\pi_n}
\rightharpoonup
\ell^{\omega,\pi}_{T}
\qquad\text{in }L^{r}(\mathbb{R}^{d}),\qquad r\in(1,\infty]
$$
where weak convergence is replaced by weak-$*$ convergence when
$r=\infty$. Let $r'=r/(r-1)$, with $r'=1$ when $r=\infty$.

The weak convergence of the measures
$L_{\omega}^{\pi_n}(x)dx$ established in Section~\ref{sec:loctime} identifies
$\ell^{\omega,\pi}_{T}$ as the density of the quadratic-variation
occupation measure:
$$
\int_{\mathbb{R}^{d}}
\varphi(x)\ell^{\omega,\pi}_{T}(x),dx
=
\int_{0}^{T}
\varphi(\omega(t)),d\operatorname{tr}[\omega]_{\pi}(t),
\qquad
\varphi\in C_c(\mathbb{R}^{d}).
$$
Equivalently,
$$
\mu_{\omega,T}^{\pi}(dx)
:=
\int_{0}^{T}
\delta_{\omega(t)}(dx),
d\operatorname{tr}[\omega]_{\pi}(t)
=
\ell^{\omega,\pi}_{T}(x),dx.
$$

For every $h\in L^{r'}(\mathbb{R}^{d})$, the seminorm of
Definition~\ref{def:W-omega-pi} has the explicit representation
$$
\begin{aligned}
\|h\|_{W^{0}_{\omega,\pi}}
&=
\limsup_{n\to\infty}
\int_{\mathbb{R}^{d}}
\|h(x)\|L_{\omega}^{\pi_n}(x) dx
\
&=
\int_{\mathbb{R}^{d}}
|h(x)|\ell^{\omega,\pi}_{T}(x) dx
=
|h|_{L^{1}(\mu_{\omega,T}^{\pi})}.
\end{aligned}
$$
Indeed, $\|h\|\in L^{r'}$ is an admissible test function for the weak
$L^{r}$ convergence.

Consequently,
$$
\{
F\in W^{2,1}(\mathbb{R}^{d})\cap C(\mathbb{R}^{d}):
\nabla^{2}F\in L^{r'}(\mathbb{R}^{d})
\}
\subseteq
W_{\omega}^{2,\pi}.
$$
To verify this inclusion, let $F_m=\eta_m*F$ be a standard
mollification. Then
$
F_m\in C_b^{2}(\mathbb{R}^{d})\cap W^{2,1}(\mathbb{R}^{d}),
$
and, since $\omega([0,T])$ is compact,
$$
\sup_{t\in[0,T]}
|F_m(\omega(t))-F(\omega(t))|
\mathop{\longrightarrow}^{m\to\infty}0.
$$
Moreover,
$$
\nabla^{2}F_m\mathop{\longrightarrow}^{m\to\infty}\nabla^{2}F
\qquad\text{in }L^{r'}(\mathbb{R}^{d}),
$$
and therefore
$$
\begin{aligned}
\|\nabla^{2}F_m-\nabla^{2}F\|_{W^{0}_{\omega,\pi}}
&=
\int_{\mathbb{R}^{d}}
|\nabla^{2}F_m-\nabla^{2}F|
,\ell^{\omega,\pi}_{T} dx
\\
&\leq
\|\ell^{\omega,\pi}_{T}\|_{L^{r}}
\|\nabla^{2}F_m-\nabla^{2}F\|_{L^{r'}}
\longrightarrow 0.
\end{aligned}
$$

Thus, on the dual space $L^{r'}$, the path-dependent Hessian topology
is exactly the weighted topology
$$
L^{1}\bigl(\ell^{\omega,\pi}_{T}(x),dx\bigr).
$$
This does not in general characterize all of
$W_{\omega}^{2,\pi}$: weak $L^{r}$ convergence of the discrete
densities only controls Hessian errors belonging to $L^{r'}$, whereas
Definition~\ref{def:W-omega-pi} allows weak Hessians that are merely in $L^{1}$. The
full space may therefore retain information about the discrete
sequence $(L_{\omega}^{\pi_n})_n$ which is not encoded by the limiting
occupation density alone.}
\end{remark}

\begin{remark}[One dimensional case]
{\em When $d=1$, the seminorm of Definition~ \ref{def:W-omega-pi} may be related to the
discrete local times of Davis--Ob{\l}\'oj--Siorpaes \cite{davis2014} as follows.
Let $\lambda_T^{n,+}$ denote the discrete local time of $\omega$
along $\pi_n$, and let $\lambda_T^{n,-}$ denote the discrete local
time of the reversed path $
\overleftarrow{\omega}(t):=\omega(T-t)
$  
along the time-reversed partition $\overleftarrow{\pi}_n$. Up to the
irrelevant choice of interval endpoints,
$$
L_{\omega}^{\pi_n}(u)
=
\frac12
\left(
\lambda_{T}^{n,+}(u)+\lambda_{T}^{n,-}(u)
\right).
$$
Indeed, for a single increment from $a$ to $b$, the forward and
reversed discrete local-time contributions are
$$
2|b-u|\mathbf{1}_{[a\wedge b,a\vee b)}(u)\quad 
{\rm and}
\quad
2|a-u|\mathbf{1}_{[a\wedge b,a\vee b)}(u),
$$
respectively. Their half-sum is
$|b-a|\mathbf{1}_{[a\wedge b,a\vee b)}(u),$
which is precisely the corresponding contribution to
$L_{\omega}^{\pi_n}$ in dimension one.

Suppose that, for some $r\in(1,\infty]$, both discrete local-time
sequences converge to the same terminal local time:
$$
\lambda_{T}^{n,+}
\rightharpoonup
\ell_{T},
\qquad
\lambda_{T}^{n,-}
\rightharpoonup
\ell_{T}
\qquad\text{in }L^{r}(\mathbb{R}),
$$
with weak-$*$ convergence when $r=\infty$. Then
$$
L_{\omega}^{\pi_n}
\rightharpoonup
\ell_{T}
\qquad\text{in }L^{r}(\mathbb{R}),
$$
and hence, with $r'=r/(r-1)$,
$$
\|h\|_{W^{0}_{\omega,\pi}}
=
\int_{\mathbb{R}}|h(u)|\ell_{T}(u),du,
\qquad
h\in L^{r'}(\mathbb{R}).
$$
Since
$
W^{2,1}(\mathbb{R})\subset C^{1}(\mathbb{R}),
$
the trace condition in Definition~\ref{def:W-omega-pi} is automatic for every
$F\in W^{2,1}(\mathbb{R})$. It follows that
$$
\{
F\in W^{2,1}(\mathbb{R}):
F''\in L^{r'}(\mathbb{R})\ \}
\subseteq
W_{\omega}^{2,\pi}.
$$
Equivalently,
$$
W_{\omega}^{2,\pi}
\cap
\{
F\in W^{2,1}(\mathbb{R}):
F''\in L^{r'}(\mathbb{R})
\}
=
\{
F\in W^{2,1}(\mathbb{R}):
F''\in L^{r'}(\mathbb{R})
\}.
$$

In the bounded local-time case,
$$
L_{\omega}^{\pi_n}
\mathop{\to}^{*}
\ell_T
\qquad\text{in }L^{\infty}(\mathbb{R}),
$$
the Banach--Steinhaus theorem gives
$$
\sup_{n\geq1}
\|L_{\omega}^{\pi_n}\|_{L^{\infty}}<\infty.
$$
For every $F\in W^{2,1}(\mathbb{R})$, standard mollification then
yields
$$
\begin{aligned}
\|F_m''-F''\|_{W^{0}_{\omega,\pi}}
&\leq
\sup_{n\geq1}
\|L_{\omega}^{\pi_n}\|_{L^{\infty}}
\|F_m''-F''\|_{L^{1}}
\
&\longrightarrow0.
\end{aligned}
$$
Since $W_{\omega}^{2,\pi}$ is, by definition, a subspace of
$W^{2,1}(\mathbb{R})$, one obtains the exact identification
$$
W_{\omega}^{2,\pi}
=
W^{2,1}(\mathbb{R}).
$$
If a  local time $(\ell_t)_{t\in[0,T]}$ exists in the
sense of \cite{davis2014} then, whenever both pathwise constructions
apply to the same function $F$, comparison of the two
change-of-variable formulae gives
$$
[F'(\omega),\omega]_{\pi}(t)
=
\int_{\mathbb{R}}\ell_t(u)F''(u),du.
$$
Accordingly,
$$
F(\omega(t))-F(\omega(0))
=
\int_{0}^{t}F'(\omega(s)),d^{\pi}\omega(s)
+
\frac12
\int_{\mathbb{R}}\ell_t(u)F''(u),du.
$$
Thus, in dimension one, the covariation term in Theorem~\ref{theorem:weakchangeofvar} agrees
with the local-time pairing of Davis--Ob{\l}\'oj--Siorpaes whenever
the two theories are simultaneously applicable.}
\end{remark}
Let us define  the seminorms
\[
\|h\|_{n, p, \omega}:=\left(   \sum_{\pi_n}\left( |h(\omega(t_i^n)|^p\,+\,|h(\omega(t_{i+1}^n)|^p\right) |t_{i+1}^n-t_i^n|\right)^{1/p}.
\]
Assume without loss of generality  ${\rm osc}(\omega, \pi_n)\leq 1$. Then from Theorem \ref{theorem:Sobineq}
\eq{\label{2eq:fLMf}
\int_{\Rr^d} |f(x)| \mathcal{L}^{\pi_n}_{\omega}(x) \d x \leq \sum_{i=0}^{m(n)-1} \left(M_1(f )(\omega(t^n_{i+1}))\,+\,M_1(f)(\omega(t^n_i))\right)|\omega(t^n_{i+1})\,-\,\omega(t^n_i)|^2
}
Define
$$
H_n(t)
:=
\sum_{i=0}^{m(n)-1}
\left(
M_1f(\omega(t_i^n))
+
M_1f(\omega(t_{i+1}^n))
\right)
\mathbf{1}_{[t_i^n,t_{i+1}^n)}(t),
$$
and
$$
\rho_{n,\omega}(t)
:=
\sum_{i=0}^{m(n)-1}
\frac{
|\omega(t_{i+1}^n)-\omega(t_i^n)|^2
}{
t_{i+1}^n-t_i^n
}
\mathbf{1}_{[t_i^n,t_{i+1}^n)}(t).
$$
Then the right-hand side of~\eqref{2eq:fLMf} is
$$
\int_0^T H_n(t)\rho_{n,\omega}(t)\,dt.
$$
Consequently, if
$$
\sup_{n\geq 1}
\|\rho_{n,\omega}\|_{L^q([0,T])}<\infty,
\qquad
\frac1p+\frac1q=1,
$$
H\"older's inequality gives
$$
\begin{aligned}
\int_{\mathbb{R}^d}|f(x)|L_\omega^{\pi_n}(x)\,dx
&\leq
\int_0^T H_n(t)\rho_{n,\omega}(t)\,dt
\leq
\|H_n\|_{L^p([0,T])}
\|\rho_{n,\omega}\|_{L^q([0,T])}.
\end{aligned}
$$
Since
$$
\|H_n\|_{L^p([0,T])}
\leq C_p\|M_1f\|_{n,p,\omega},
$$
we obtain
\begin{equation}
\|f\|_{W^0_{\omega,\pi}}
\leq
C_{p,\omega}
\limsup_{n\to\infty}
\|M_1f\|_{n,p,\omega},
\qquad
C_{p,\omega}
:=
C_p\sup_{n\geq1}
\|\rho_{n,\omega}\|_{L^q([0,T])}.
\label{2eq:normMf}
\end{equation}
Lemma~\ref{lemma:qvar} shows that the required uniform bound holds almost surely
when $\omega$ is a Brownian path and the partitions satisfy
$\sum_n|\pi_n|<\infty$.
 
We will use this estimate \eqref{2eq:normMf} in Theorem \ref{lemma:Sobinc} to prove the inclusion of certain Sobolev spaces in $\mathcal{W}^{2, \pi}_{\omega}$ when $\omega $ is a  typical   Brownian path.

\section{Change of variable formulas}\label{sec:main}
\subsection{Pathwise Ito formula for weakly differentiable functions}
We are now ready to prove our  main result.

\begin{theorem}[Pathwise integration and weak change-of-variable formula]\label{theorem:weakchangeofvar}
Let
$\omega\in C([0,T];\mathbb{R}^d)\cap Q_{\pi}([0,T];\mathbb{R}^d)$
satisfy \emph{(6)}, and let
$F\in\mathcal{W}_{\omega,\pi}^{2}$.
Let
$\{F_m\}_{m\geq 1}\subset C_b^2(\mathbb{R}^d)\cap
W^{2,1}(\mathbb{R}^d)$
be any sequence satisfying Definition~\ref{def:W-omega-pi}, namely

\begin{equation}
F_m(\omega(t))\longrightarrow F(\omega(t)),
\qquad t\in[0,T],
\label{eq:pathwise-approximation-F}
\end{equation}
\begin{equation}
\left\|
\nabla^2F_m-\nabla^2F
\right\|_{W_{\omega,\pi}^{0}}
\longrightarrow 0.
\label{eq:pathwise-approximation-Hessian}
\end{equation}
For every $t\in[0,T]$, the limits
\begin{equation}
\lim_{m\to\infty}
\int_0^t
\nabla F_m(\omega(s))\,d^{\pi}\omega(s)
\label{eq:smooth-integral-limit}
\end{equation}
\begin{equation}
{\rm and}\qquad [\nabla F(\omega),\omega]_{\pi}(t)
:=
\lim_{m\to\infty}
[\nabla F_m(\omega),\omega]_{\pi}(t)
\label{eq:covariation-by-approximation}
\end{equation}
exist and are independent of the sequence $(F_m)_{m\geq 1}$ satisfying
\eqref{eq:pathwise-approximation-F} and
\eqref{eq:pathwise-approximation-Hessian}.
We  define the integral of $\nabla F(\omega)$ with
respect to $\omega$ along $\pi$ by
\begin{equation}
\int_0^t
\nabla F(\omega(s))\,d^{\pi}\omega(s)
:=
\lim_{m\to\infty}
\int_0^t
\nabla F_m(\omega(s))\,d^{\pi}\omega(s).
\label{eq:pathwise-integral-by-approximation}
\end{equation}
This definition does not require the weak gradient $\nabla F$ to have a
finite pointwise representative at the partition points. Then
\begin{equation}
F(\omega(t))-F(\omega(0))
=
\int_0^t
\nabla F(\omega(s))\,d^{\pi}\omega(s)
+
\frac{1}{2}
[\nabla F(\omega),\omega]_{\pi}(t).
\label{eq:pathwise-Ito-approximation}
\end{equation}
\end{theorem}
\begin{proof}

For $H\in C_b^2(\mathbb{R}^d)\cap W^{2,1}(\mathbb{R}^d)$, using the second inequality in Theorem \eqref{theorem:Sobinq} twice with $x=\omega(t^n_i)), y= \omega(t^n_{i+1})$ and with reversed roles  $x= \omega(t^n_{i+1}),  y=\omega(t^n_i))$
we get
\[
|\left(\nabla H(\omega(t^n_{k+1}))- \nabla H(\omega(t^n_k))\right)\cdot \left(\omega(t^n_{k+1}) - \omega(t^n_k)\right)|\leq 
C_d\int_{\Rr^d} \nabla^2H(x) \mathcal{L}^{n, i}_{\omega}(x) dx.
\]

Summing up these inequalities over the partition points in the interval $[0,t]$, and taking the limit gives 

\begin{equation}
\left|
[\nabla H(\omega),\omega]_{\pi}(t)
\right|
\leq
C_d
\left\|
\nabla^2H
\right\|_{W_{\omega,\pi}^{0}}.
\label{eq:smooth-covariation-estimate}
\end{equation}

Applying this estimate to $H=F_m-F_k$ gives

\begin{equation}
\begin{aligned}
\bigl|
[\nabla F_m(\omega),\omega]_{\pi}(t)
-
[\nabla F_k(\omega),\omega]_{\pi}(t)
\bigr|
&\leq
C_d
\left\|
\nabla^2F_m-\nabla^2F_k
\right\|_{W_{\omega,\pi}^{0}}                                                        \\
\leq
C_d
\left\|
\nabla^2F_m-\nabla^2F
\right\|_{W_{\omega,\pi}^{0}}  &       +
C_d
\left\|
\nabla^2F_k-\nabla^2F
\right\|_{W_{\omega,\pi}^{0}}.
\end{aligned}
\label{eq:covariation-Cauchy-estimate}
\end{equation}

By \eqref{eq:pathwise-approximation-Hessian}, the right-hand side tends
to zero as $m,k\to\infty$. Hence

\begin{equation}
\left\{
[\nabla F_m(\omega),\omega]_{\pi}(t)
\right\}_{m\geq 1}
\end{equation}
is a Cauchy sequence. This proves the existence of the limit in
\eqref{eq:covariation-by-approximation}.

For every $m$, the pathwise Itô formula for the smooth function $F_m$
gives

\begin{equation}
F_m(\omega(t))-F_m(\omega(0))
=
\int_0^t
\nabla F_m(\omega(s))\,d^{\pi}\omega(s)
+
\frac{1}{2}
[\nabla F_m(\omega),\omega]_{\pi}(t).
\label{eq:smooth-Ito-formula-Fm}
\end{equation}

Equivalently,

\begin{equation}
\int_0^t
\nabla F_m(\omega(s))\,d^{\pi}\omega(s)
=
F_m(\omega(t))-F_m(\omega(0))
-
\frac{1}{2}
[\nabla F_m(\omega),\omega]_{\pi}(t).
\label{eq:smooth-integral-representation}
\end{equation}

The first two terms on the right-hand side converge by
\eqref{eq:pathwise-approximation-F}, while the covariation term converges
by \eqref{eq:covariation-Cauchy-estimate}. It follows that the limit in
\eqref{eq:smooth-integral-limit} exists.

It remains to prove that the limit does not depend on the choice of the
approximating sequence. Let $\{G_k\}_{k\geq 1}$ be another sequence
satisfying Definition~\ref{def:W-omega-pi} for the same function $F$.
Applying \eqref{eq:smooth-covariation-estimate} to $F_m-G_k$ yields

\begin{equation}
\begin{aligned}
&\bigl|
[\nabla F_m(\omega),\omega]_{\pi}(t)
-
[\nabla G_k(\omega),\omega]_{\pi}(t)
\bigr|\\
&\leq
C_d
\left\|
\nabla^2F_m-\nabla^2G_k
\right\|_{W_{\omega,\pi}^{0}} \\
&\leq
C_d
\left\|
\nabla^2F_m-\nabla^2F
\right\|_{W_{\omega,\pi}^{0}}+
C_d
\left\|
\nabla^2G_k-\nabla^2F
\right\|_{W_{\omega,\pi}^{0}}.
\end{aligned}
\label{eq:cross-covariation-estimate}
\end{equation}

The right-hand side tends to zero as $m,k\to\infty$. In addition,

\begin{equation}
(F_m-G_k)(\omega(t))\longrightarrow 0
\quad{\rm and}\qquad
(F_m-G_k)(\omega(0))\longrightarrow 0.
\end{equation}
Using the smooth Itô formula for $F_m-G_k$, we obtain

\begin{equation}
\begin{aligned}
&
\left|
\int_0^t
\nabla F_m(\omega(s))\,d^{\pi}\omega(s)
-
\int_0^t
\nabla G_k(\omega(s))\,d^{\pi}\omega(s)
\right|                                                             \\
&\qquad\leq
\left|(F_m-G_k)(\omega(t))\right|
+
\left|(F_m-G_k)(\omega(0))\right|                                   \\
&\qquad\quad+
\frac{1}{2}
\left|
[\nabla F_m(\omega),\omega]_{\pi}(t)
-
[\nabla G_k(\omega),\omega]_{\pi}(t)
\right|.
\end{aligned}
\label{eq:cross-integral-estimate}
\end{equation}

Every term on the right-hand side tends to zero as $m,k\to\infty$.
Therefore the limits obtained from $\{F_m\}$ and $\{G_k\}$ coincide.
This proves that
\eqref{eq:pathwise-integral-by-approximation} and
\eqref{eq:covariation-by-approximation} are well defined.
Finally, passing to the limit $m\to\infty$ in
\eqref{eq:smooth-Ito-formula-Fm} gives
\eqref{eq:pathwise-Ito-approximation}.
\end{proof}

\subsection{Sobolev trace along Brownian paths}
The next result justifies the definition of the space $\mathcal{W}^2_{\omega, \pi}$, showing in particular that it contains non-smooth functions for a class of $\omega$ which are typical paths of a Brownian motion.

Recall the definition of the Sobolev-Slobodeckij space $W^{2+s,p}(\Rr^d )$:
\[ 
W^{2+s,p}(\Rr^d ):=\left\{f\in W^{2 ,p}(\Rr^d): \left(\int_{\Rr^d} \int_{\Rr^d} \frac{|\nabla^2 f(x)\,-\,\nabla^2f(y)|^p}{|x-y|^{s p + d}} \; dx \; dy\right)^{\frac{1}{p}}   <\infty \right\}
\]
and let
\[ W^{2+, p}(\Rr^d):=\mathop{\cup}_{s\in (0, 1)}W^{2+s,p}(\Rr^d). \]
In this section we will require two assumptions on the sequence of partitions:
\begin{hyp}[Balanced partitions]\label{ass:partitioncond}
\eq{
M_1:=\sup_{n}\frac {|\pi_n|}{ |\underline{\pi}_n| }<+\infty,
}
where $|\pi_n|:=max_{i} (t^n_{i+1} - t^n_i),\, |\underline{\pi}_n|:=min_{i} (t^n_{i+1} - t^n_i)$.
\end{hyp}
\begin{hyp}\label{ass:partitioncond2}
There exists a constant $M_2>1$ such that
\eq{
1<M_2\leq \frac {|\pi_n|}{ |\pi_{n+1}| }\qquad    \, \forall n\geq 1.
}
\end{hyp}
\begin{theorem}
\label{lemma:Sobinc}
Let $(\Omega,B,P_x)$ be the canonical Wiener space as in Section \ref{sec:Bessel} and   $\pi=(\pi_n)_{n\geq1}$ be a sequence of partitions of $[0,T]$ satisfying Assumptions~\ref{ass:partitioncond} and \ref{ass:partitioncond2}. Let
$$
F \in W^{2+,p}(\mathbb{R}^d)\cap W^{2,1}(\mathbb{R}^d),
\qquad 1<p<\infty.
$$
Then there exists a polar set $E\subset\mathbb{R}^d$ such that, for every starting point
$x_0\notin E$, 
$$
\mathbb{P}_{x_0}\left(\ F\in W^2_{B,\pi}\  \right)=1.
$$
Equivalently, for $x_0\notin E$, there exists a sequence $(F_m)_{m\geq1}\subset C_b^2(\mathbb{R}^d)\cap W^{2,1}(\mathbb{R}^d)$ such that
$$
F_m(B(t))\mathop{\to}^{m\to\infty} F(B(t)),\qquad t\in[0,T],
$$
and
$$
\mathbb{P}_{x_0}\left(\  \|\nabla^2F_m-\nabla^2F\|_{W^0_{B,\pi}}\mathop{\to}^{m\to\infty} 0\right)=1.
$$
\end{theorem}
 The proof is given in Appendix \ref{sec.Proof54}.

\subsection{Relation with the F\"{o}llmer-Protter formula}

Finally, we can recover a version of the F\"ollmer-Protter formula for functions of multidimensional Brownian motion \cite{follmerprotter2000}:
\begin{proposition}
Let $\pi=(\pi_n)_{n\geq 1}$ satisfy Assumptions~1 and~2 and let
\[
F\in W^{2+,p}(\mathbb{R}^d)
   \cap W^{2,1}(\mathbb{R}^d)
   \cap W_{\mathrm{loc}}^{1,2}(\mathbb{R}^d),
\qquad 1<p<\infty.
\]
Then, outside a polar set $E$ of starting points, the pathwise integral
constructed in Theorem~\ref{theorem:weakchangeofvar} agrees almost surely with the stochastic
It\^{o} integral, and
\[
\forall x\notin E,\quad 
\mathbb{P}_{x}\left( F(B(t))-F(B(0))
=
\int_0^t \nabla F(B(s))\,dB(s)
+
\frac{1}{2}[\nabla F(B),B]_{\pi}(t)\right)=1
\]
for every $t\in[0,T]$.\label{prop.FP}
\end{proposition}
For $d=1$ we obtain a pathwise derivation of the F\"ollmer-Protter-Shiryaev formula \cite{FPS1995}.
\begin{proof}
Let $(F_m)_{m\geq 1}$ be the deterministic smooth approximating sequence obtained in the proof of Theorem~\ref{lemma:Sobinc}. Thus there exists a polar set $E_0\subset\mathbb{R}^d$ such that, for every $x_0\notin E_0$,
$$
F_m(B(t))\longrightarrow F(B(t)),
\qquad t\in[0,T],
$$
and
$$
\left\|\nabla^2F_m-\nabla^2F\right\|_{W^0_{B,\pi}}
\longrightarrow 0
$$
$\mathbb{P}_{x_0}$-almost surely.
Since $F_m$ are standard mollifications of $F\in W^{1,2}_{\mathrm{loc}}(\mathbb{R}^d)$,
$$
\nabla F_m\longrightarrow\nabla F
\quad\text{in }L^2_{\mathrm{loc}}(\mathbb{R}^d).
$$
Passing to a further deterministic subsequence, and relabelling it again by $(F_m)_{m\geq 1}$, we may assume that, for every integer $R\geq 1$,
$$
\sum_{m=1}^{\infty}
\left\|\nabla F_m-\nabla F\right\|_{L^2(B_R)}^2
<\infty,
$$
where $B_R:=B(0,R)$. Indeed, the subsequence can be chosen inductively so that
$$
\left\|\nabla F_m-\nabla F\right\|_{L^2(B_j)}^2
\leq 2^{-m},
\qquad 1\leq j\leq m.
$$
Passing to this subsequence does not affect the two conclusions obtained from Theorem \ref{lemma:Sobinc}.

Choose a finite-valued Borel representative of $\nabla F$. For $R\geq 1$, define
$$
H_R(x):=
\mathbf{1}_{B_R}(x)
\sum_{m=1}^{\infty}
\left|\nabla F_m(x)-\nabla F(x)\right|^2.
$$
Then $H_R\in L^1(\mathbb{R}^d)$. By Lemma~B.1, there exists a polar set $E_R^1$ such that
$$
\int_{\mathbb{R}^d}
H_R(x)\nu(|x-x_0|)\,dx
<\infty
$$
for every $x_0\notin E_R^1$. Similarly, since $\nabla F\in L^2_{\mathrm{loc}}(\mathbb{R}^d)$, the function
$$
G_R(x):=
\mathbf{1}_{B_R}(x)|\nabla F(x)|^2
$$
belongs to $L^1(\mathbb{R}^d)$. Hence there exists a polar set $E_R^2$ such that
$$
\int_{\mathbb{R}^d}
G_R(x)\nu(|x-x_0|)\,dx
<\infty
$$
for every $x_0\notin E_R^2$.
Let
$$
E_1:=
\bigcup_{R=1}^{\infty}
\left(E_R^1\cup E_R^2\right).
$$
Then $E_1$ is polar. Fix $x_0\notin E_0\cup E_1$ and define
$$
\tau_R:=
\inf\{t\geq 0:|B(t)|\geq R\}.
$$
Using the Brownian transition density and estimate~(46), we obtain
$$
\begin{aligned}
E_{x_0}\left[
\int_0^{T\wedge\tau_R}
|\nabla F(B(s))|^2\,ds
\right]
&\leq
\int_{B_R}
|\nabla F(x)|^2
\left(
\int_0^T p_s(x-x_0)\,ds
\right)dx
\\
&\leq
C_{R,x_0}
\int_{B_R}
|\nabla F(x)|^2
\nu(|x-x_0|)\,dx
<\infty.
\end{aligned}
$$
Consequently,
$$
\int_0^{T\wedge\tau_R}
|\nabla F(B(s))|^2\,ds
<\infty
$$
$\mathbb{P}_{x_0}$-almost surely. Therefore the stopped stochastic integral
$$
J^R(t):=
\int_0^{t\wedge\tau_R}
\nabla F(B(s))\,dB(s)
$$
is well defined. The family $(J^R)_{R\geq 1}$ is consistent under stopping, and since $\tau_R\uparrow\infty$ almost surely, it defines the local It\^o integral
$$
J(t):=
\int_0^t\nabla F(B(s))\,dB(s),
\qquad t\in[0,T].
$$

We next identify the pathwise integral for every smooth approximant. For $s\in[t_i^n,t_{i+1}^n)$, set
$ \eta_n(s):=t_i^n,$
and define
$$
S_n^m(t):=
\sum_{i=0}^{m(n)-1}
\nabla F_m(B(t_i^n))\cdot
\left(
B(t_{i+1}^n\wedge t)-B(t_i^n\wedge t)
\right).
$$
Since the integrand is a predictable step process,
$$
S_n^m(t)=
\int_0^t
\nabla F_m(B(\eta_n(s)))\,dB(s).
$$
Let
$$
J_m(t):=
\int_0^t\nabla F_m(B(s))\,dB(s).
$$
Because $F_m\in C_b^2(\mathbb{R}^d)$, the function $\nabla F_m$ is globally Lipschitz. The Burkholder--Davis--Gundy inequality gives
$$
\begin{aligned}
E_{x_0}\left[
\sup_{0\leq t\leq T}
|S_n^m(t)-J_m(t)|^2
\right]
&\leq
C
E_{x_0}\left[
\int_0^T
\left|
\nabla F_m(B(\eta_n(s)))
-
\nabla F_m(B(s))
\right|^2 ds
\right]
\\
&\leq
C\|\nabla^2F_m\|_{\infty}^2
E_{x_0}\left[
\int_0^T
|B(\eta_n(s))-B(s)|^2\,ds
\right]
\\
&\leq
C_{d,T}\|\nabla^2F_m\|_{\infty}^2|\pi_n|.
\end{aligned}
$$
Assumption~\ref{ass:partitioncond2} implies
$$
\sum_{n=1}^{\infty}|\pi_n|<\infty.
$$
Hence, for each fixed $m$,
$$
\sum_{n=1}^{\infty}
E_{x_0}\left[
\sup_{0\leq t\leq T}
|S_n^m(t)-J_m(t)|^2
\right]
<\infty.
$$
By Tonelli's theorem,
$$
\sum_{n=1}^{\infty}
\sup_{0\leq t\leq T}
|S_n^m(t)-J_m(t)|^2
<\infty
$$
almost surely. Therefore,
$$
\sup_{0\leq t\leq T}
|S_n^m(t)-J_m(t)|
\longrightarrow 0
$$
almost surely. Since $m$ ranges over a countable set, this convergence holds simultaneously for all $m$ outside one null set. It follows that
$$
\int_0^t
\nabla F_m(B(s))\,\mathrm{d}^{\pi}B(s)
=
\int_0^t
\nabla F_m(B(s))\,dB(s)
=
J_m(t)
$$
for every $m$ and every $t\in[0,T]$, almost surely.
It remains to prove that $J_m$ converges almost surely and uniformly to $J$. For fixed $R\geq 1$, define
$$
D_m^R(t):=
\int_0^{t\wedge\tau_R}
\left(
\nabla F_m(B(s))-\nabla F(B(s))
\right)dB(s).
$$
By the Burkholder--Davis--Gundy inequality and Tonelli's theorem,
$$
\begin{aligned}
\sum_{m=1}^{\infty}
E_{x_0}\left[
\sup_{0\leq t\leq T}
|D_m^R(t)|^2
\right]
&\leq
C
E_{x_0}\left[
\int_0^{T\wedge\tau_R}
\sum_{m=1}^{\infty}
\left|
\nabla F_m(B(s))-\nabla F(B(s))
\right|^2 ds
\right]
\\
&\leq
C
\int_{B_R}
H_R(x)
\left(
\int_0^T p_s(x-x_0)\,ds
\right)dx
\\
&\leq
C_{R,x_0}
\int_{B_R}
H_R(x)\nu(|x-x_0|)\,dx <\infty.
\end{aligned}
$$
It follows again from Tonelli's theorem that
$$
\sum_{m=1}^{\infty}
\sup_{0\leq t\leq T}
|D_m^R(t)|^2
<\infty
$$
almost surely. Hence
$$
\sup_{0\leq t\leq T}
|D_m^R(t)|
\longrightarrow 0
$$
almost surely.

Taking the countable intersection over $R\geq 1$, we may assume that the preceding convergence holds for every integer $R$. Since Brownian paths are bounded on $[0,T]$, almost every path satisfies $\tau_R>T$ for some integer $R$. For such an $R$,
$$
D_m^R(t)=J_m(t)-J(t),
\qquad t\in[0,T].
$$
Therefore,
$$
\sup_{0\leq t\leq T}
|J_m(t)-J(t)|
\longrightarrow 0
$$
almost surely.
By Theorem~\ref{theorem:weakchangeofvar},
$$
\int_0^t
\nabla F(B(s))\,\mathrm{d}^{\pi}B(s)
=
\lim_{m\to\infty}
\int_0^t
\nabla F_m(B(s))\,\mathrm{d}^{\pi}B(s).
$$
Using the identification for the smooth approximants and the almost-sure uniform convergence of $J_m$ to $J$, we obtain
$$
\begin{aligned}
\int_0^t
\nabla F(B(s))\,\mathrm{d}^{\pi}B(s)
&=
\lim_{m\to\infty}
\int_0^t
\nabla F_m(B(s))\,dB(s)
\\
&=
\int_0^t
\nabla F(B(s))\,dB(s)
\end{aligned}
$$
for every $t\in[0,T]$, almost surely.
Finally, Theorem~\ref{theorem:weakchangeofvar} gives
$$
F(B(t))-F(B(0))
=
\int_0^t
\nabla F(B(s))\,\mathrm{d}^{\pi}B(s)
+
\frac{1}{2}
[\nabla F(B),B]_{\pi}(t).
$$
Substituting the preceding identity yields
$$
F(B(t))-F(B(0))
=
\int_0^t
\nabla F(B(s))\,dB(s)
+
\frac{1}{2}
[\nabla F(B),B]_{\pi}(t)
$$
for every $t\in[0,T]$, almost surely.

The exceptional set
$$
E:=E_0\cup E_1
$$
is polar. This proves the proposition.
\end{proof}

\begin{remark}{\em
The additional assumption
$F\in W_{\mathrm{loc}}^{1,2}(\mathbb{R}^d)$
is not stronger than the regularity condition used in the
F\"ollmer-Protter formula \cite{follmerprotter2000}; it is essentially the same Sobolev regularity
assumption needed to define the classical stochastic It\^o integral
$\int_0^t \nabla F(B(s)),dB(s)$.
In dimension one, this corresponds to requiring that $F$ be locally
absolutely continuous with $F'\in L_{\mathrm{loc}}^2(\mathbb{R})$.

The stronger regularity assumption in the present result is instead
$$
F\in W^{2+,p}(\mathbb{R}^d)\cap W^{2,1}(\mathbb{R}^d),
\qquad
W^{2+,p}(\mathbb{R}^d)
=
\bigcup_{s\in(0,1)}W^{2+s,p}(\mathbb{R}^d),
$$
which is required for the pathwise construction. This assumption does
not automatically imply
$F\in W_{\mathrm{loc}}^{1,2}(\mathbb{R}^d)$
for every $p>1$ and every dimension $d$. A sufficient Sobolev embedding
condition is
$$
1+s>d\left(\frac{1}{p}-\frac{1}{2}\right)
$$
for some $s\in(0,1)$ such that
$F\in W^{2+s,p}(\mathbb{R}^d)$. In particular, the implication holds
automatically when $p\geq2$.}
\end{remark}

\begin{appendices}
    \section{$L^q$ estimates for quadratic Brownian sums}
\label{sec:appA}
The following estimate is the only property of the discrete quadratic variation
of Brownian motion used in the proof of Theorem~\ref{lemma:Sobinc}.
For a continuous path $\omega\colon[0,T]\to\mathbb{R}^d$, define
\begin{equation}
\rho_{n,\omega}(t)
:=
\sum_{i=0}^{m(n)-1}
\frac{\left|\omega(t_{i+1}^n)-\omega(t_i^n)\right|^2}
     {t_{i+1}^n-t_i^n}
\mathbf{1}_{[t_i^n,t_{i+1}^n)}(t).
\label{eq:discrete-qv-density}
\end{equation}

\begin{lemma}
\label{lemma:qvar}
Let $(\Omega,B,\mathbb{P}_x)$ be the canonical Wiener space as in Section~\ref{sec:Bessel}, 
$q>1$, and $\pi=(\pi_n)_{n\geq1}$ a sequence of partitions of
$[0,T]$ such that
$$
\sum_{n=1}^{\infty}|\pi_n|<\infty.
$$
Then
$$
\|\rho_{n,B}\|_{L^q([0,T])}^q
\longrightarrow
c_{q,d}T
\qquad \mathbb{P}_x\text{-almost surely},
$$
where
$$
c_{q,d}:=\mathbb{E}|Z|^{2q},
\qquad Z\sim N(0,I_d).
$$
In particular,
$$
\sup_{n\geq1}\|\rho_{n,B}\|_{L^q([0,T])}<\infty
\qquad \mathbb{P}_x\text{-almost surely}.
$$
\end{lemma}

\begin{proof}
Write
$$
\Delta_i^n:=t_{i+1}^n-t_i^n,
\qquad
Z_i^n:=\frac{B(t_{i+1}^n)-B(t_i^n)}{\sqrt{\Delta_i^n}}.
$$
For each $n$, the random variables $(Z_i^n)_{i=0}^{m(n)-1}$ are
independent and have distribution $N(0,I_d)$. Therefore,
$$
\|\rho_{n,B}\|_{L^q([0,T])}^q
=
\sum_{i=0}^{m(n)-1}|Z_i^n|^{2q}\Delta_i^n.
$$
Let
$b_{q,d}:=\operatorname{Var}(|Z|^{2q}),
\qquad Z\sim N(0,I_d).$
Since $\sum_i\Delta_i^n=T$, we have
$$
\mathbb{E}_x\!\left[\|\rho_{n,B}\|_{L^q([0,T])}^q\right]
=c_{q,d}T,
$$
and, by independence,
\begin{align*}
\operatorname{Var}_x\!\left(
\|\rho_{n,B}\|_{L^q([0,T])}^q
\right)
&=
 b_{q,d}\sum_{i=0}^{m(n)-1}(\Delta_i^n)^2 \leq b_{q,d}T|\pi_n|.
\end{align*}
Hence, for every $\varepsilon>0$, Chebyshev's inequality gives
$$
\sum_{n=1}^{\infty}
\mathbb{P}_x\!\left(
\left|
\|\rho_{n,B}\|_{L^q([0,T])}^q-c_{q,d}T
\right|>\varepsilon
\right)
\leq
\frac{b_{q,d}T}{\varepsilon^2}
\sum_{n=1}^{\infty}|\pi_n|
<\infty.
$$
The Borel--Cantelli lemma yields the asserted almost-sure convergence.
The uniform $L^q$ bound follows immediately.
\end{proof}

\section{Technical lemmas}
\label{sec:appB}
 
Let $(\Omega, B, \mathbb{P}_x)$ be the canonical Wiener space defined in Section \ref{sec:Bessel}. We take a sequence of compact sets $K_l\subset \Rr^d,\, l\geq 1$ such that for any starting point $x_0$ outside a polar set $E$
\eq{\label{2eq:Kl}
\lim_{l\to\infty}P( B(t)\in K_l,\quad \forall t\in [0, T]) =1.
}
We consider partitions $\pi:=\{\pi_n\}$  of $[0, T]$, satisfying  Assumptions \ref{ass:partitioncond} and \ref{ass:partitioncond2}.
We note that under the latter condition, we have
\eq{\label{ass:partitioncond3}
\sum_{n\geq 1} |\pi_n| <+\infty,
}
\eq{\label{2eq:partsum}
\frac 1 {|\pi_1|}\,+\,\frac 1 {|\pi_2|}\,+\ldots\,+ \frac 1 {|\pi_n|}\leq M_3 \frac 1 {|\pi_n|}.
}
Indeed, under Assumption \ref{ass:partitioncond2}, we have 
$$|\pi_k|\leq \frac 1 {M_2} |\pi_{k-1}|\leq \ldots \leq  \frac 1 {M_2^{k-1}} |\pi_1|$$
thus 
\[
\sum_{n\geq 1} |\pi_n| \leq |\pi_1| \sum_{n\geq 1}  \frac 1 {M_2^{n-1}}<+\infty.
\]
similarly 
$$\frac 1 {|\pi_k|}\leq \frac 1 {M_2} \frac 1 {|\pi_{k+1}|}\leq \ldots \leq  \frac 1 {M_2^{n-k}} |\pi_n|,$$
and
\[
\frac 1 {|\pi_1|}\,+\,\frac 1 {|\pi_2|}\,+\ldots\,+ \frac 1 {|\pi_n|}\leq \left( \frac 1 {M_2^{n-1}}\,+\,\frac 1 {M_2^{n-2}}\,+ \dots+\, 1\right)\frac 1 {|\pi_n|} \leq \frac {M_2}{M_2 -1} \frac 1 {|\pi_n|}.
\]
Next, let $p_{s}(x)$ be the Gaussian distribution with mean $0$ and variance $s$:
\[
p_s(x):= \frac 1 {(2\pi s)^{\frac d 2}} e^{-\frac {|x|^2}{2 s}}.
\]
We will use the following estimates (see e.g. \cite{follmerprotter2000})
\eq{\label{2eq:psx}
\int_0^T p_{_s}(x) \d s \leq C(R) \nu (|x|), \forall |x|\leq R
}
where 
\begin{equation}\label{eq.nu}
\nu(r)=
\begin{cases}
1, & d=1,\\
1+\log^+(1/r), & d=2,\qquad \\
r^{2-d}, & d\ge 3.
\end{cases}
\end{equation}
Similarly, we have 
\eq{\label{2eq:psxs}
\int_0^T \frac{p_{_s}(x) }{s}\d s \leq C\, \frac 1 {|x|^{d}}.
}
From the assumptions on the partitions, we have
$$
\sup_{n, i} \frac{t^n_{i+1}}{t^n_{i}}\leq 1+ \sup_{n, i} \frac{t^n_{i+1}- t^n_i}{t^n_{i}- t^n_{i-1}}<+\infty
$$
thus as in \cite{follmerprotter2000}
\eq{\label{2eq:ptnix}
\sum_{i=0}^{m(n)-1} p_{_{t^n_i}}(x)(t^n_{i+1}- t^n_i) \leq C(R) \nu (|x|), \forall |x|\leq R.
}
We prove two auxiliary lemmas. The first lemma implies the existence of singular integrals with respect to the kernel $\nu$ outside polar sets of Brownian measure.

\begin{lemma}
\label{lemma:plorset}
Let $h\in L^1(\mathbb R^d)$. There  exists a polar set
$E\subset\mathbb R^d$ such that
$$
\int_{\mathbb R^d}|h(x)|\ \nu(|x-x_0|)\,dx<\infty
$$
for every $x_0\notin E$, where $\nu$ is defined in \eqref{eq.nu}.
\end{lemma}
\begin{proof}
Let $g_2$ denote the Bessel kernel of order $2$. 
For $d=1$, the assertion follows immediately from $\nu\equiv 1$.
We may therefore assume $d\geq 2$.
From the behaviour of $g_2$ at the origin given in~\eqref{2eq:Bessassymp}, and from the positivity and continuity of $g_2$ away from the origin, there exists a constant $C_d>0$ such that
$$
\nu(|z|)
\leq
C_d\bigl(1+g_2(z)\bigr),
\qquad z\in\mathbb{R}^d\setminus\{0\}.
$$
Indeed, near the origin this follows from the asymptotic behaviour of $g_2$; for $|z|\geq 1$, the function $\nu(|z|)$ is bounded.
Consequently,
\begin{eqnarray*}
\int_{\mathbb{R}^d}
|h(x)|\nu(|x-x_0|)\ dx
&\leq
C_d\int_{\mathbb{R}^d}|h(x)| dx +
C_d\int_{\mathbb{R}^d}
|h(x)|g_2(x-x_0)\,dx \\
&=
C_d|h|_{L^1(\mathbb{R}^d)}
+
C_d\bigl(g_2*|h|\bigr)(x_0).
\end{eqnarray*}
It therefore remains to show that the Bessel potential $g_2*|h|$ is finite outside a $(2,1)$-polar set.
For $N\geq 1$, define
$$
E_N:=
\{
x_0\in\mathbb{R}^d:
\bigl(g_2* |h|\bigr)(x_0)>N
\}.
$$
On $E_N$ one has
$$
g_2*\left(\frac{|h|}{N}\right)>1.
$$
Hence, by the definition of the Bessel $(2,1)$-capacity,
$$ B_{2,1}(E_N)
\leq
\frac{1}{N}\|h\|_{L^1(\mathbb R^d)}.$$
Now let
$$
E:=
\{
x_0\in\mathbb{R}^d:
\bigl(g_2*|h|\bigr)(x_0)=+\infty
\}.
$$
Since $E\subset E_N$ for every $N\geq 1$, monotonicity of the Bessel capacity gives
$$
B_{2,1}(E)
\leq
B_{2,1}(E_N)
\leq
\frac{1}{N}\|h\|_{L^1(\mathbb{R}^d)}.
$$
Letting $N\to\infty$, we obtain
$ B_{2,1}(E)=0.$
Thus $E$ is $(2,1)$-polar. By  ~\cite[Theorem~1]{AM}, every $(2,1)$-polar set is $(1,2)$-polar. Since $(1,2)$-polar sets are polar for Brownian motion, $E$ is polar in the sense of Definition~2.8.

For every $x_0\notin E$, one has
$
\bigl(g_2*|h|\bigr)(x_0)<\infty,
$
and therefore
$$
\int_{\mathbb{R}^d}
|h(x)|\nu(|x-x_0|)\,dx<\infty.
$$
This proves the result.
\end{proof}


The next lemma provides conditions for functions to belong to $\mathcal{W}^0_{\omega, \pi}$.
\begin{lemma} \label{lem:B2}
\label{lemma:Lpinc}
Let $(\Omega, B, \mathbb{P}_x)$ be the canonical Wiener space defined as in Section \ref{sec:Bessel}
  and $\pi=(\pi_n)_{n\geq1}$ be a sequence of partitions of $[0,T]$ satisfying Assumptions \eqref{ass:partitioncond}, \eqref{ass:partitioncond2}.  Let $(h_m)_{m\geq1}\subset L^1(\mathbb{R}^d)$ be a sequence of nonnegative finite-valued Borel functions  such that:
\begin{enumerate}
\item $\|h_m\|_{L^1(\mathbb{R}^d)}\to 0$ as $m\to\infty$;
\item for every compact set $K\subset\mathbb{R}^d$,
$$
\int_K\int_K \frac{|h_m(x)-h_m(y)|}{|x-y|^d}\,dx\,dy<\infty.
$$
\end{enumerate}
Then there exists a polar set $E\subset\mathbb{R}^d$ and a subsequence $(h_{m_k})_{k\geq1}$ such that, for every $x_0\notin E$,
$$
\lim_{k\to\infty}\ \limsup_{n\to\infty}\|h_{m_k}\|_{n,1,B}=0,
\qquad \mathbb{P}_{x_0}\text{-a.s.},
$$
where
$$
\|h\|_{n,1,B}
:=
\int_0^T
\sum_{i=0}^{m(n)-1}
\bigl(|h(B(t_i^n))|+|h(B(t_{i+1}^n)|)\bigr)(t_{i+1}^n-t_{i}^n)
$$
\end{lemma}
\begin{proof}
Write
$$
\Delta_i^n:=t_{i+1}^n-t_i^n,\qquad
\delta_n:=|\pi_n|=\max_i\Delta_i^n,\qquad
\underline{\delta}_n:=\min_i\Delta_i^n.
$$
By Assumption 1,
$$
\frac{\delta_n}{\underline{\delta}_n}\leq M_1,
$$
and, by Assumption 2, the sequence $(\delta_n)_{n\geq 1}$ decreases at least geometrically.
We choose finite-valued Borel representatives of all the functions $h_m$, redefining them on Lebesgue-null sets if necessary.
Since $\|h_m\|_{L^1(\mathbb{R}^d)}\to 0$, we may choose a subsequence, still denoted by $(h_{m_k})_{k\geq 1}$, such that
$$
\sum_{k=1}^{\infty}\|h_{m_k}\|_{L^1(\mathbb{R}^d)}<\infty.
$$
Let
$$
\widetilde{h}:=\sum_{k=1}^{\infty}h_{m_k} \in L^1(\mathbb{R}^d).
$$
Notice that $\widetilde{h}$ need not have compact support.
For $R\geq 1$, let
$$
K_R:=\overline{B(0,R)},\qquad {\rm and}\quad 
\Omega_R:={B([0,T])\subset K_R}.
$$
Using continuity of Brownian paths,
$$
\mathbb{P}_{x_0}\left(\bigcup_{R\geq 1}\Omega_R\right)=1.
$$
We first prove that
$$
\int_0^T\widetilde{h}(B(t))\,dt<\infty
$$
almost surely, outside a polar set of starting points.
For every $R\geq 1$, the function $\widetilde{h}\mathbf{1}_{K_R}$ belongs to $L^1(\mathbb{R}^d)$ and has compact support. By Lemma \ref{lemma:plorset} there exists a polar set $E_R^0$ such that
$$
\int_{K_R}\widetilde{h}(x)\ \nu(|x-x_0|)\,dx<\infty
$$
for every $x_0\notin E_R^0$.
Using estimate (46), we obtain
\begin{eqnarray}
    \mathbb{E}_{x_0}\left[
\int_0^T
\widetilde{h}(B(t))\mathbf{1}_{K_R}(B(t))\,dt
\right]
&=
\int_{K_R}\widetilde{h}(x)
\left(\int_0^T p_t(x-x_0)\,dt\right)dx \nonumber\\
&\leq
C_R
\int_{K_R}\widetilde{h}(x)\nu(|x-x_0|)\,dx
<\infty.
\end{eqnarray}
Consequently,
$$
\int_0^T
\widetilde{h}(B(t))\mathbf{1}_{K_R}(B(t))\,dt
<\infty
$$
almost surely under $\mathbb{P}_{x_0}$.
Let
$$
E^0:=\bigcup_{R\geq 1}E_R^0.
$$
The set $E^0$ is polar. For every $x_0\notin E^0$, it follows that
$$
\int_0^T\widetilde{h}(B(t))\,dt<\infty
$$
almost surely, because almost every Brownian path belongs to $\Omega_R$ for some $R$.

We next prove the convergence of the Riemann sums. Fix $k\geq 1$ and write $h:=h_{m_k}$. Define
$$
L_n(h):=
\sum_{i=0}^{m(n)-1}
h(B(t_i^n))\Delta_i^n.
$$
For the intervals not beginning at zero, set
$$
A_n(h):=
\sum_{i=1}^{m(n)-1}
\int_{t_i^n}^{t_{i+1}^n}
\left(h(B(t_i^n))-h(B(t))\right)\,dt.
$$

Fix $R\geq 1$ and define
$$
G_{h,R}(x):=
\mathbf{1}_{K_R}(x)
\int_{K_R}
\frac{|h(x)-h(y)|}{|x-y|^d}\,dy.
$$
By the second assumption of the lemma,
$$
\int_{\mathbb{R}^d}G_{h,R}(x)\,dx
= \int_{K_R}\int_{K_R}
\frac{|h(x)-h(y)|}{|x-y|^d}\,dy\,dx
<\infty.
$$
Thus $G_{h,R}\in L^1(\mathbb{R}^d)$ and has compact support. By Lemma B.1, there exists a polar set $E_{h,R}$ such that
$$
\int_{K_R}G_{h,R}(x)\nu(|x-x_0|)\,dx<\infty
$$
for every $x_0\notin E_{h,R}$.
On the event $\Omega_R$, the Markov property and the joint density of $W(t_i^n)$ and $B(t)$ give
$$
\begin{aligned}
\mathbb{E}_{x_0}\left[
\mathbf{1}_{\Omega_R}|A_n(h)|
\right]
&\leq
\int_{K_R}\int_{K_R}
|h(x)-h(y)| 
\sum_{i=1}^{m(n)-1}
p_{t_i^n}(x-x_0)
\left(
\int_0^{\Delta_i^n}p_s(y-x)\,ds
\right)
dy\,dx.
\end{aligned}
$$

Since $\Delta_i^n\leq\delta_n$ and
$$
\Delta_i^n\geq\underline{\delta}_n
\geq\frac{\delta_n}{M_1},
$$
we have
$$
\begin{aligned}
&\sum_{i=1}^{m(n)-1}
p_{t_i^n}(x-x_0)
\int_0^{\Delta_i^n}p_s(y-x)\,ds \
&\leq
\frac{M_1}{\delta_n}
\left(
\sum_{i=1}^{m(n)-1}
p_{t_i^n}(x-x_0)\Delta_i^n
\right)
\int_0^{\delta_n}p_s(y-x)\,ds.
\end{aligned}
$$
By estimate \eqref{2eq:ptnix},
$$
\sum_{i=1}^{m(n)-1}
p_{t_i^n}(x-x_0)\Delta_i^n
\leq C_R\nu(|x-x_0|)
$$
for $x\in K_R$. Therefore,
$$
\begin{aligned}
\mathbb{E}_{x_0}\left[
\mathbf{1}_{\Omega_R}|A_n(h)|
\right]
&\leq
C_R
\int_{K_R}\int_{K_R}
|h(x)-h(y)|\nu(|x-x_0|) \
&\times
\frac{1}{\delta_n}
\left(
\int_0^{\delta_n}p_s(y-x)\,ds
\right)\,
dy\,dx.
\end{aligned}
$$

We now use the geometric decay of the mesh sizes. By Tonelli's theorem,
$$
\begin{aligned}
\sum_{n=1}^{\infty}
\frac{1}{\delta_n}
\int_0^{\delta_n}p_s(z)\,ds
&=
\int_0^{\delta_1}
p_s(z)
\sum_{{n:,\delta_n\geq s}}
\frac{1}{\delta_n}\,ds.
\end{aligned}
$$
For $s\in(0,\delta_1]$, let $N=N(s)$ be the largest integer such that $\delta_N\geq s$. By estimate (45),
$$
\sum_{{n:,\delta_n\geq s}}
\frac{1}{\delta_n}
=
\sum_{n=1}^{N}\frac{1}{\delta_n}
\leq
\frac{M_3}{\delta_N}
\leq
\frac{M_3}{s}.
$$
Consequently, estimate \eqref{2eq:psxs} yields
$$
\begin{aligned}
\sum_{n=1}^{\infty}
\frac{1}{\delta_n}
\int_0^{\delta_n}p_s(z)\,ds
&\leq
C\int_0^T\frac{p_s(z)}{s}\,ds \
&\leq
\frac{C}{|z|^d}.
\end{aligned}
$$
It follows that
\begin{eqnarray}
\sum_{n=1}^{\infty}
\mathbb{E}_{x_0}\left[
\mathbf{1}_{\Omega_R}|A_n(h)|
\right]
&\leq
C_R
\int_{K_R}\int_{K_R}
\frac{|h(x)-h(y)|}{|x-y|^d}
\nu(|x-x_0|)\,dy\,dx \nonumber\\
&=
C_R
\int_{K_R}
G_{h,R}(x)\nu(|x-x_0|)\,dx
<\infty.\nonumber
\end{eqnarray}
Hence, by Tonelli's theorem,
$$
\sum_{n=1}^{\infty}
\mathbf{1}_{\Omega_R}|A_n(h)|<\infty
$$
almost surely. In particular,
$ A_n(h)\longrightarrow 0 $
almost surely on $\Omega_R$.
It remains to treat the interval $[0,t_1^n]$, which was omitted because $W(0)=x_0$ is deterministic. Its contribution is
$$
\int_0^{t_1^n}
\left(h(x_0)-h(B(t))\right)dt.
$$
Since $t_1^n\leq\delta_n\to 0$, $h(x_0)<\infty$, and
$$
\int_0^T h(B(t))\,dt
\leq
\int_0^T\widetilde{h}(B(t))\,dt
<\infty,
$$
this contribution converges to zero almost surely. Therefore,
$$
L_n(h)\mathop{\longrightarrow}^{n\to\infty}
\int_0^T h(B(t))\,dt
$$
almost surely.
Taking a countable union over $k\geq 1$ and $R\geq 1$, we obtain a polar set $E^1$ such that, for every $x_0\notin E^1$, the preceding convergence holds simultaneously for all $h_{m_k}$.
Finally, by Assumption \ref{ass:partitioncond}
$$
\begin{aligned}
|h|_{n,1,B}
&=
\sum_{i=0}^{m(n)-1}
\left(
h(B(t_i^n))+h(B(t_{i+1}^n))
\right)\Delta_i^n \
&\leq
(1+M_1)L_n(h)+\delta_n h(B(T)).
\end{aligned}
$$
Thus,
$$
\limsup_{n\to\infty}\|h_{m_k}\|_{n,1,B}
\leq
(1+M_1)
\int_0^T h_{m_k}(B(t))\,dt.
$$
Moreover,
$$
\begin{aligned}
\sum_{k=1}^{\infty}
\int_0^T h_{m_k}(B(t))\,dt
&=
\int_0^T
\sum_{k=1}^{\infty}h_{m_k}(B(t))\,dt \
&=
\int_0^T\widetilde{h}(B(t))\,dt
<\infty
\end{aligned}
$$
almost surely. Hence,
$$
\int_0^T h_{m_k}(B(t))\,dt\longrightarrow 0.
$$
It follows that
$$
\lim_{k\to\infty}
\limsup_{n\to\infty}
|h_{m_k}|_{n,1,B}
=0
$$
almost surely.
The exceptional set
$ E:=E^0\cup E^1$
is polar, being a countable union of polar sets. This completes the proof.
\end{proof}

    \section{Proof of Theorem \ref{lemma:Sobinc}}\label{sec.Proof54}
\begin{proof}
Fix $s\in(0,1)$ such that $F\in W^{2+s,p}(\mathbb{R}^d)$. Let
$ F_m:=\eta_m*F,$
where $\eta_m(x)=m^d\eta(mx)$ and $\eta\in C_c^\infty(\mathbb{R}^d)$ is a standard mollifier. Since $F\in W^{2,1}(\mathbb{R}^d)$, we have
$$
F_m\in C_b^2(\mathbb{R}^d)\cap W^{2,1}(\mathbb{R}^d).
$$
Moreover,
$$
F_m(x)\mathop{\longrightarrow}^{m\to\infty} F(x)
$$
for every Lebesgue point $x$ of $F$.

Let $E_1$ be the set of non-Lebesgue points of $F$.  
Since $p>1$ and
\[
    F\in W^{2+s,p}(\mathbb R^d)
    \subset W^{2,p}(\mathbb R^d),
\]
 Theorem~\ref{theorem:Lebesgue} implies that $E_1$ is
$(2,p)$-polar. Since $2p>2$, the inclusion theorem for
Bessel exceptional classes \cite[Theorem~1]{AM}
implies that $E_1$ is $(1,2)$-polar, and hence polar for
Brownian motion (Def. \ref{chp2def:polar}).

Consequently, for every $x_0\notin E_1$, Brownian motion
starting from $x_0$ almost surely does not enter $E_1$ at
positive times. Since $x_0$ itself is a Lebesgue point of
$F$, it follows that
\[
    F_m(B(t))\longrightarrow F(B(t)),
    \qquad t\in[0,T],
\]
$\mathbb{P}_{x_0}$-almost surely.
 
It remains to prove convergence of the Hessians in the seminorm $\|\cdot\|_{W^0_{B,\pi}}$. Define
$$
u_m(x):=
\|\nabla^2F_m(x)-\nabla^2F(x)\|.
$$
Let
$$g^*_m(x)=M_1(|u_m|)(x), \quad N_m=\{x,\quad g^*_m(x)=+\infty\},$$
where $M_1$ is the truncated Hardy--Littlewood maximal operator introduced in Section~\ref{sec:Sobspace}. 
Since $u_m\in L^p(\mathbb{R}^d),\ p>1$, $N_m$ is a Lebesgue null set. We use the finite-valued Borel representatives
$$g_m(x)=g^*_m(x)\ 1_{N_m^c},\quad h_m(x):=g_m(x)^p.$$
These modifications do not change the $L^p$ or $W^{s,p}$ classes. However the maximal function estimate may change at the initial point $B(0)=x_0$ so the first partition interval will be treated separately.

Since mollification converges in $W^{2+s,p}(\mathbb{R}^d)$ and the matrix norm is a Lipschitz function,
$$
\|u_m\|_{L^p(\mathbb{R}^d)}
+
[u_m]_{W^{s,p}(\mathbb{R}^d)}
\longrightarrow 0.
$$
The operator $M_1$ is bounded on $L^p(\mathbb{R}^d)$. It is also bounded on $W^{s,p}(\mathbb{R}^d)$. For completeness, let $\tau_z f(x):=f(x+z)$. Translation invariance and sublinearity of $M_1$ give
$$
|\tau_z(M_1f)-M_1f|
\leq
M_1(|\tau_zf-f|)
$$
almost everywhere. Hence
$$
\|\tau_z(M_1f)-M_1f\|_{L^p(\mathbb{R}^d)}
\leq
C_p\|\tau_zf-f\|_{L^p(\mathbb{R}^d)}.
$$

Using translation invariance and sublinearity of
$M_1$ together with
 the translation characterization of the Sobolev--Slobodeckij
seminorm we obtain
$$
\begin{aligned}
[M_1f]_{W^{s,p}(\mathbb{R}^d)}^p
&=
\int_{\mathbb{R}^d}
\frac{
\left\|\tau_z(M_1f)-M_1f\right\|_{L^p(\mathbb{R}^d)}^p
}{
|z|^{d+sp}
}\,dz
\\
&\leq
C_p^p
\int_{\mathbb{R}^d}
\frac{
\left\|\tau_zf-f\right\|_{L^p(\mathbb{R}^d)}^p
}{
|z|^{d+sp}
}\,dz
=
C_p^p[f]_{W^{s,p}(\mathbb{R}^d)}^p.
\end{aligned}
$$
Therefore,
$$
[M_1f]_{W^{s,p}(\mathbb{R}^d)}
\leq
C_p[f]_{W^{s,p}(\mathbb{R}^d)}.
$$
Consequently,
$$
\|g_m\|_{L^p(\mathbb{R}^d)}
+
[g_m]_{W^{s,p}(\mathbb{R}^d)}
\mathop{\longrightarrow}^{m\to\infty} 0,\qquad 
\|h_m\|_{L^1(\mathbb{R}^d)}
=
\|g_m\|_{L^p(\mathbb{R}^d)}^p
\mathop{\longrightarrow}^{m\to\infty} 0.
$$
We next verify the second hypothesis of Lemma~B.2. Let $K\subset\mathbb{R}^d$ be compact and let $p'=p/(p-1)$. 
Since $g_m\geq 0$,
$$
|g_m(x)^p-g_m(y)^p|
\leq
p\left(
g_m(x)^{p-1}+g_m(y)^{p-1}
\right)
|g_m(x)-g_m(y)|.
$$
By H\"older's inequality on $K\times K$,
$$
\begin{aligned}
\int_K\int_K
\frac{|h_m(x)-h_m(y)|}{|x-y|^d}\,dx\,dy
&\quad\leq
p
\left(
\int_K\int_K
\frac{
\left(
g_m(x)^{p-1}+g_m(y)^{p-1}
\right)^{p'}
}{
|x-y|^{d-sp'}
}\,dx\,dy
\right)^{1/p'}
\\
&\qquad\times
\left(
\int_K\int_K
\frac{
|g_m(x)-g_m(y)|^p
}{
|x-y|^{d+sp}
}\,dx\,dy
\right)^{1/p}.
\end{aligned}
$$
Since $sp'>0$, the kernel $|x-y|^{-d+sp'}$ is locally integrable. Thus
$$
\begin{aligned}
&\int_K\int_K
\frac{
\left(
g_m(x)^{p-1}+g_m(y)^{p-1}
\right)^{p'}
}{
|x-y|^{d-sp'}
}\,dx\,dy
\quad\leq
C_{K,p,s}\|g_m\|_{L^p(K)}^p.
\end{aligned}
$$
It follows that
$$
\int_K\int_K
\frac{|h_m(x)-h_m(y)|}{|x-y|^d}\,dx\,dy
\leq
C_{K,p,s}
\|g_m\|_{L^p(K)}^{p-1}
[g_m]_{W^{s,p}(K)}
<\infty.
$$
Hence the sequence $(h_m)_{m\geq1}$ satisfies both assumptions of Lemma~B.2.

Lemma~B.2 therefore yields a subsequence $(m_k)_{k\geq1}$ and a polar set $E_2\subset\mathbb{R}^d$ such that, for every $x_0\notin E_2$,
$$
\lim_{k\to\infty}
\limsup_{n\to\infty}
\|h_{m_k}\|_{n,1,B}
=
0,
$$
$\mathbb{P}_{x_0}$-almost surely. We now pass to this subsequence and, to simplify notation, write again $F_m$, $u_m$, $g_m$, and $h_m$.
Since $u_m\in L^1(\mathbb{R}^d),$ Lemma \ref{lemma:plorset} gives  for every $m$ a polar set $E_{3,m}$ such that for every $x_0\notin E_{3,m}$
$$ \int_{\mathbb{R}^d} |u_m(z)|\ \nu(|z-x_0|)\ dz < \infty.$$
Then the set
$E_3=\cup_{m\geq 1} E_{3,m}$ is polar. Because positive partition times $t^n_i>0$ for $i,n\geq 1$ form a countable (deterministic) set and Brownian motion has density at all positive times, $$
\mathbb{P}_{x_0}\left(
B(t_i^n)\notin\bigcup_{m\geq1}N_m
\text{ for every }n\geq1,\ 1\leq i\leq m(n)
\right)=1.
$$
Thus the maximal-function estimates remain valid at every positive
partition point $t^n_i>0$.

Let $q=p/(p-1)$.  By Assumption~\ref{ass:partitioncond2},
$$
|\pi_n|
\leq
|\pi_1|M_2^{-(n-1)},\quad{\rm so}\quad
\sum_{n=1}^{\infty}|\pi_n|<\infty.
$$
So by Lemma~\ref{lemma:qvar},
$$
\sup_{n\geq1}
\|\rho_{n,B}\|_{L^q([0,T])}
<\infty
$$
almost surely. 
Fix $x_0\notin E_3$. Write $b^n_i=B(t^n_i), r^n_i=B(t^n_{i+1})-B(t^n_i)$ and let $Q^n_0$ be the ball associated with the first increment $r^n_0$.
For $i\geq 1$ the argument leading to \eqref{2eq:fLMf} gives
$$ \int_{\mathbb{R}^d} |u_m(z)|L^{n,i}_B(z) dz\leq C_d \|r^n_i\|^2(g_m(b^n_i)+ g_m(b^n_{i+1}))$$ 
For $i=0$ retain the kernel centred at $x_0$ explicitly
$$ \int_{\mathbb{R}^d} |u_m(z)|L^{n,0}_B(z) dz\leq C_d J_{m,n}(x_0)+ C_d \|r^n_0\|^2g_m(b^n_1)$$
where $$J_{m,n}(x_0)= \int_{Q^n_0} \frac{|u_m(z)|}{|z-x_0|^{d-2}}dz$$
Since $Q^n_0\subset B(x_0,|r^n_0|)$ shrinks to zero as $n\to \infty$ and $\|z-x_0\|^{2-d}\leq C_d \nu(\|z-x_0\|)$ on small balls, the definition of $E_3$ implies $J_{m,n}(x_0)\to 0$ for every fixed $m$. 
Define
$$
\widetilde H_{m,n}(t)
:=
g_m(b_1^n)\mathbf{1}_{[t_0^n,t_1^n)}(t)
+
\sum_{i=1}^{m(n)-1}
\left(
g_m(b_i^n)+g_m(b_{i+1}^n)
\right)
\mathbf{1}_{[t_i^n,t_{i+1}^n)}(t).
$$
Summing over $i$ gives
$$
\int_{\mathbb{R}^d}|u_m(z)|L_B^{\pi_n}(z)\,dz
\leq
C_dJ_{m,n}(x_0)
+
C_d\int_0^T\widetilde H_{m,n}(t)\rho_{n,B}(t)\,dt,
$$
Since
$$
\|\widetilde H_{m,n}\|_{L^p([0,T])}
\leq
C_p\|g_m\|_{n,p,B}
=
C_p\|h_m\|_{n,1,B}^{1/p},
$$
applying H\"older's inequality to all remaining endpoint terms and using Lemma \ref{lemma:qvar} we obtain
$$ \|u_m\|_{W^0_{B,\pi}}\leq C_{p,B}\left( \mathop{\lim\sup}_{n\to\infty} \|h_m\|_{n,1,B}\right)^{1/p}$$
The conclusion of Lemma~B.2 therefore implies
$$
\forall x_0\notin E_2\cup E_3,\quad \mathbb{P}_{x_0}\left(\ \|
\nabla^2F_m-\nabla^2F
\|_{W^0_{B,\pi}}
\longrightarrow 0\ \right)=1
$$
Combining this convergence with
$$
F_m(B(t))\mathop{\longrightarrow}^{m\to\infty} F(B(t)),
\qquad t\in[0,T],
$$
we conclude that
$$
\forall x_0\notin E_1\cup E_2\cup E_3,\quad \mathbb{P}_{x_0}\left(\ F\in W^2_{B,\pi}\ \right)=1.
$$
Finally,
$
E:=E_1\cup E_2 \cup E_3$
is polar. This proves the theorem.
\end{proof}

\end{appendices}

\end{document}